\documentclass{amsart}
\usepackage{amssymb}
\usepackage{graphicx}
\usepackage{amscd}
\usepackage{stmaryrd}
\usepackage{accents,array,color,url}

\allowdisplaybreaks

\newcommand{\ignore}[1]{}

\numberwithin{figure}{section}
\numberwithin{table}{section}

\newcommand\tr{\operatorname{tr}}

\renewcommand\div{\operatorname{div}}
\newcommand\curl{\operatorname{curl}}
\newcommand\rot{\operatorname{rot}}

\newcommand\vol{\mathsf{vol}}

\newcommand\R{\mathbb{R}}
\newcommand\eps{e}

\newcommand\E{{\mathcal E}}

\newcommand\Lin{\mathcal L}

\renewcommand\P{{\mathcal P}}

\newcommand\T{{\mathcal T}}

\newcommand{\<}{\langle}
\renewcommand{\>}{\rangle}
\newcommand{\0}{\mathring}

\numberwithin{equation}{section}

\newtheorem{thm}{Theorem}[section]

\newtheorem{lem}[thm]{Lemma}

\begin{document}

\title{Local bounded cochain projections}
\author{Richard S. Falk}
\address{Department of Mathematics,
Rutgers University, Piscataway, NJ 08854}
\email{falk@math.rutgers.edu}
\urladdr{http://www.math.rutgers.edu/\char'176falk/}
\thanks{}
\author{Ragnar Winther}
\address{Centre of Mathematics for Applications
and Department of Mathematics,
University of Oslo, 0316 Oslo, Norway}
\email{ragnar.winther@cma.uio.no}
\urladdr{http://heim.ifi.uio.no/\char'176rwinther/}
\thanks{}
\date{November 19, 2012}
\thanks{The work of the first author was supported in part by NSF grant
DMS-0910540. The work of the second author was supported by the Norwegian
Research Council.}

\begin{abstract}
We construct projections from $H \Lambda^k(\Omega)$, the space of differential
$k$ forms on $\Omega$ which belong to $L^2(\Omega)$ and whose exterior
derivative also belongs to $L^2(\Omega)$, to finite dimensional subspaces of
$H \Lambda^k(\Omega)$ consisting of piecewise polynomial differential forms
defined on a simplicial mesh of $\Omega$.  Thus, their definition requires
less smoothness than assumed for the definition of the canonical
interpolants based on the degrees of freedom. Moreover, these projections have
the properties that they commute with the exterior derivative and are bounded
in the $H \Lambda^k(\Omega)$ norm independent of the mesh size $h$.  Unlike
some other recent work in this direction, the projections are also locally
defined in the sense that they are defined by local operators on overlapping
macroelements, in the spirit of the Cl\'ement interpolant.
\end{abstract}
\maketitle

\section{Introduction}
\label{sec:intro}
Projection operators which commute with the governing differential
operators are key tools for the stability analysis of finite element
methods associated to a differential complex. In fact, such
projections have been a central feature of the analysis of mixed
finite element methods since the beginning of such analysis,
cf. \cite{Brezzi, BrezziFortin}. However, a key difficulty is that,
for most of the standard finite element spaces, the canonical
projection operators defined from the degrees of freedom are not
well--defined on the appropriate function spaces. This is the case for
the Lagrange finite elements, considered as a subspace of the Sobolev
space $H^1$, and for the Raviart-Thomas \cite{Raviart-Thomas},
Brezzi-Douglas-Marini \cite{Brezzi-Douglas-Marini}, and N\'ed\'elec
\cite{Nedelec1,Nedelec2} finite element spaces considered as subspaces
of $H(\div)$ or $H(\curl)$.  For example, the classical continuous
piecewise linear interpolant, based on the values at the vertices of
the mesh, is not defined for functions in $H^1$ in dimensions higher
than one. Therefore, even if the canonical projections commute with
the governing differential operators on smooth functions, these
operators cannot be directly used in a stability argument for the
associated finite element method due to the lack of boundedness of the
projections in the proper operator norms.  In addition to the
canonical projection operators, it is worth mentioning another family
of projection operators that commute with the exterior derivative.
This approach, usually referred to as projection based interpolation, is detailed in
the work of Demkowicz and collaborators (cf. \cite{Cao-Demkowicz},
\cite{Demkowicz}, \cite{Demkowicz-Babuska}, \cite {Demkowicz-Buffa},
\cite{Demkowicz-Kurtz}). The main motivation for the construction of
these operators was the analysis of the so called p--version of the
finite element method, i.e., the focus is on the dependence of the
polynomial degree of the finite element spaces. 
However, as in the case of the canonical projection
operators, the definition of these operators requires some additional
smoothness of the underlying functions, so again they cannot be used
directly in the standard stability arguments.
On the other hand, the classical Cl\'ement interpolant \cite{clement}
is a local operator, and it is well--defined for functions in
$L^2$. However, the Cl\'ement interpolant is not a projection, and the
obvious extensions of the Cl\'ement operator to higher order finite
element differential forms (cf. \cite{acta,bulletin}) do not commute
with the exterior derivative. Therefore, these operators are not
directly suitable for a stability analysis.

Bounded commuting projections have been constructed in previous
work. The first such construction was given by Sch\"{o}berl in
\cite{schoberl}. The idea is to compose a smoothing operator and the
unbounded canonical projection to obtain a bounded operator which maps
the proper function space into the finite element space. In order to
obtain a projection, one composes the resulting operator with the
inverse of this operator restricted to the finite element
space. In \cite{schoberl}, a perturbation of the finite element space
itself was used to construct the proper smoother.
In a related paper, Christiansen \cite{snorre2007} proposed to use a more standard
smoothing operator defined by a mollifier function. Using this idea, 
variants of Sch\"{o}berl's construction are analyzed in \cite[Section 5]{acta},
\cite[Section 5]{bulletin}, and \cite{c-w}.   
The constructed
projections commute with the exterior derivative and they are bounded
in $L^2$. Therefore, they can be used to establish stability of finite
element methods.  However, these projections lack another key property
of the canonical projections; they are not locally defined. In fact,
up to now it has been an open question if it is possible to construct
bounded and commuting projections which are locally defined. The
projections defined in this paper have all these properties. The
construction presented below resembles the construction of the
Cl\'ement operator in the sense that it is based on local operators on
overlapping macroelements.

We will adopt the language of finite element exterior calculus as in
\cite{acta,bulletin}.  The theory presented in these papers may be
described as follows. Let $\Omega \subset \R^n$ be a bounded
polyhedral domain, and let $H\Lambda^k(\Omega)$ be the space of
differential $k$ forms $u$ on $\Omega$, which is in $L^2$, and where
its exterior derivative, $du = d^ku$, is also in $L^2$.  This space is
a Hilbert space.  The $L^2$ version of the de Rham complex then takes
the form
 \[
H\Lambda^0(\Omega)
\xrightarrow{d} H\Lambda^1(\Omega) \xrightarrow{d}
\cdots \xrightarrow{d}
H\Lambda^n(\Omega).
\]
The basic construction in finite element exterior calculus is
of a corresponding subcomplex
\[
\Lambda_h^0\xrightarrow{d} \Lambda_h^1 \xrightarrow{d}
\cdots \xrightarrow{d}
\Lambda_h^n,
\]
where the spaces $\Lambda_h^k$ are finite dimensional subspaces of
$H\Lambda^k(\Omega)$ consisting of piecewise polynomial differential forms
with respect to a partition, $\T_h$, of the domain $\Omega$.
In the theoretical
analysis of the stability of numerical methods constructed from this discrete
complex, bounded projections $\pi_h^k : H\Lambda^k(\Omega) \to \Lambda_h^k$ are
utilized, such that the diagram
\begin{equation*}
\begin{CD}
H\Lambda^0(\Omega) @>d>> H\Lambda^1(\Omega) @>d>>
\cdots @>d>> H\Lambda^n(\Omega) \\
 @VV\pi_h^0V @VV\pi_h^1V @.  @VV\pi_h^nV \\
\Lambda^0_h @>d>>
\Lambda^1_h @>d>>
\cdots @>d>> \Lambda^n_h 
\end{CD}
\end{equation*}
commutes. Such commuting projections are referred to as cochain
projections.  The importance of bounded cochain projections is
immediately seen from the analysis of the mixed finite element
approximation of the associated Hodge Laplacian.  In fact, it follows
from the results of \cite[Section 3.3]{bulletin} that the existence of
bounded cochain projections is equivalent to stability of the
associated finite element method.  Furthermore, if these projections
are local, like the ones we construct here, then improved properties
with respect to error estimates and adaptivity may be obtained.

For a general reference to finite element exterior calculus, we refer
to the survey papers \cite{acta,bulletin}, and references given
therein.  As is shown there, the spaces $\Lambda^k_h$ are taken from
two main families.  Either $\Lambda^k_h$ is of the form
$\P_r\Lambda^k(\T_h)$, consisting of all elements of
$H\Lambda^k(\Omega)$ which restrict to polynomial $k$-forms of degree
at most $r$ on each simplex $T$ in the partition $\T_h$, or
$\Lambda^k_h=\P^-_r\Lambda^k(\T_h)$, which is a space which sits
between $\P_r\Lambda^k(\T_h)$ and $\P_{r-1}\Lambda^k(\T_h)$ (the exact
definition will be recalled below).  These spaces are generalizations
of the Raviart-Thomas and Brezzi-Douglas-Marini spaces, used to
discretize $H(\div)$ and $H(\rot)$ in two space dimensions, and the
N\'ed\'elec edge and face spaces of the first and second kind, used to
discretize $H(\curl)$ and $H(\div)$ in three space dimensions.

A main feature of the construction of the projections given below is
that they are based on a direct sum geometrical decomposition of the
finite element space. In the general case of finite element
differential forms, such a decomposition was constructed in
\cite{decomp}. However, this is a standard concept in the case of
Lagrange finite elements. Let $\T_h$ be a simplicial triangulation of
a polyhedral domain $\Omega \in \R^n$. If $T$ is a simplex we let
$\Delta(T)$ be the set of all subsimplexes of $T$, and by
$\Delta_m(T)$ all subsimplexes of dimension $m$. So if $T$ is a
tetrahedron in $\R^3$, then $\Delta_m(T)$ are the set of vertices,
edges, and faces of $T$ for $m=0,1,2$, respectively.  We further
denote by $\Delta(\T_h)$ the set of all subsimplices of all dimensions
of the triangulation $\T_h$, and correspondingly by $\Delta_m(\T_h)$
the set of all subsimplices of dimension $m$.  The desired geometric
decomposition of the spaces $\P_r\Lambda^k(\T_h)$ and
$\P_r^-\Lambda^k(\T_h)$ is based on the property that the elements of
these spaces are uniquely determined by their trace, $\tr_f$, for all
$f$ of $\Delta(\T_h)$ with dimension greater or equal to $k$.  The
decompositions of the spaces $\P_r\Lambda^k(\T_h)$ established in
\cite{decomp}, is then of the form
\begin{equation}\label{decomp-basic}
P_r\Lambda^k(\T_h) = \bigoplus_{\stackrel{f \in
  \Delta(\T_h)}{\dim f \ge k}}E_{f,r}^k(\0\P_r\Lambda^k(f)).
\end{equation}
Here $\0\P_r\Lambda^k(f))$ is the subspace of $\P_r\Lambda^k(f))$
consisting of elements with vanishing trace on the boundary of $f$.
The operator $E_{f,r}^k : \0\P_r\Lambda^k(f) \to P_r\Lambda^k(\T_h)$
is an extension operator in the sense that $\tr_f\circ E_{f,r}^k$ is
the identity operator on $\0\P_r\Lambda^k(f))$. Furthermore,
$E_{f,r}^k$
is local in the sense that the support of functions in
$E_{f,r}^k(\0\P_r\Lambda^k(f))$ is restricted to  the union of the
elements of $\T_h$
which have $f$ as a subsimplex. A completely analog decomposition
\begin{equation}\label{decomp-basic-}
P_r^-\Lambda^k(\T_h) = \bigoplus_{\stackrel{f \in
  \Delta(\T_h)}{\dim f \ge k}}E_{f,r}^{k-}(\0\P_r^-\Lambda^k(f))
\end{equation}
exists for the space $P_r^-\Lambda^k(\T_h)$.

We will utilize modifications of the decompositions
\eqref{decomp-basic}
and
\eqref{decomp-basic-} to construct local bounded cochain projections
onto the finite element spaces $P_r\Lambda^k(\T_h)$
and $P_r^-\Lambda^k(\T_h)$. In the spirit of the Cl\'ement operator
we will use local projections to define the operators  $\tr_f \circ
\pi_h^k$ for each $f \in \Delta(\T_h)$ with
dimension greater or equal to $k$. To make sure that the projections $\pi_h^k$
commute with the exterior derivative we will use
a local Hodge Laplace problem to define the local projections,
while the extension operators will be of the form of
harmonic extension operators.

This paper is organized as follows. In Section~\ref{notat-prelim} we
introduce some basic notation, and we show how to construct the new
projection in the case of scalar valued functions, or zero forms.  We
also review some basic results on differential forms and their finite
element approximations.  A key step of the theory below is to
construct a special projection into the space of Whitney forms
\cite{whitney}, i.e., the space $\P_1^-\Lambda^k(\T_h)$. In fact, in
the present setting the construction in this lowest order case is in
some sense the most difficult part of the theory, since here we need
to relate local operators defined on different subdomains.
To achieve this we utilize a structure which resembles the \v{C}ech-de
Rham double complex, cf. \cite{Bott-Tu}. 
In addition
to being a projection onto the Whitney forms, the special projection
constructed in Section~\ref{sec:whitney} will also satisfy a mean
value property with respect to higher order finite element spaces,
cf. equation \eqref{R-key-prop} below.  The general construction of
the cochain projections, covering all spaces of the form
$P_r\Lambda^k(\T_h)$ or $P_r^-\Lambda^k(\T_h)$, is then performed in
Section~\ref{general}. Finally, in Section~\ref{bounds} we derive
precise local bounds for the constructed projections.

\section{Notation and preliminaries}
\label{notat-prelim}
We will use $\< \cdot, \cdot \>$ to denote $L^2$ inner products on
the domain $\Omega$. For subdomains $D \subset \Omega$ we will use a subscript
to indicate the domain, i.e., we write
$\< u, v \>_{D}$ to denote $L^2$ inner product on the domain $D$.

We will assume that $\{\T_h \}$ is a family of simplicial triangulations
of $\Omega \in \R^n$, indexed by the mesh parameter
$h = \max_{T \in \T_h} h_T$, where $h_T$ is the diameter of $T$.
In fact, $h_f$ will be used to denote the diameter of any $f \in \Delta(\T_h)$.
We will assume throughout that the triangulation is
shape regular, i.e. the ratio $h_T^n/|T|$ is uniformly bounded
for all the simplices $T \in \T_h$ and all triangulations of the family.
Here $|T|$ denotes the volume of $T$. Note that it is a simple
consequence of
shape regularity that the ratio $h_T/h_f$, for $f \in \Delta(T)$ with
$\dim f \ge 1$ is also uniformly bounded. 
We will use $[x_0,x_1, \ldots x_k]$ to denote the convex combination 
of the points $x_0,x_1, \ldots ,x_k \in \Omega$. 
Hence, any $f \in \Delta_k(\T_h)$ is of the form $f = [x_0,x_1, \ldots
x_k]$, where
$x_0,x_1, \ldots ,x_k \in \Delta_0(\T_h)$. Furthermore, the order of
the points $x_j$ reflects  the orientation of the manifold $f$.
We will let $f_j \in \Delta_{k-1}(\T_h)$ denote the subcomplex of $f$ 
obtained by deleting 
the vertex $x_j$, i.e., $f_j = [x_0, \ldots, x_{j-1}, \hat x_j,
x_{j+1}, \cdots x_{k}]$. Here the symbol $\ \widehat{} \ $ over a term 
means that the term is omitted.
Hence, if $j$ is even then $f_j$ has the orientation induced from $f$, while
if the orientation is reversed if $j$ is odd. 

For each $f \in \Delta(\T_h)$, we let $\Omega_f$ be the
associated macroelement consisting of the union of the elements of $\T_h$
containing $f$, i.e.,
\[
\Omega_f = \bigcup \{T \, | \, T \in \T_h, \, f \in \Delta(T) \, \}.
\]


\begin{center}
Fig. 1: Vertex macroelement, $n=2$. \qquad \qquad Edge macroelement, $n=2$.
\end{center}
\vskip.75in
\setlength{\unitlength}{0.35cm}
\begin{center}
\begin{picture}(30,0)
\put(0,0){\line(1,0){15}}
\put(0,0){\line(1,1){5}}
\put(0,0){\line(1,-1){5}}
\put(5,5){\line(1,0){5}}
\put(5,-5){\line(1,0){5}}
\put(15,0){\line(-1,1){5}}
\put(15,0){\line(-1,-1){5}}
\put(7.5,0){\line(-1,2){2.5}}
\put(7.5,0){\line(1,2){2.5}}
\put(7.5,0){\line(-1,-2){2.5}}
\put(7.5,0){\line(1,-2){2.5}}
\put(20,0){\line(1,-1){5}}
\put(20,0){\line(1,1){5}}
\put(20,0){\line(1,0){10}}
\put(30,0){\line(-1,1){5}}
\put(30,0){\line(-1,-1){5}}
\end{picture}
\end{center}
\vskip.75in
In addition to macroelements $\Omega_f$ we will also find it
convenient to introduce
the notion of an extended macroelement $\Omega_f^{\eps}$
defined for $f \in \Delta(\T_h)$ by
\[
\Omega_f^{\eps} = \bigcup_{g \in \Delta_{0}(f)}\Omega_g.
\]
\begin{center}
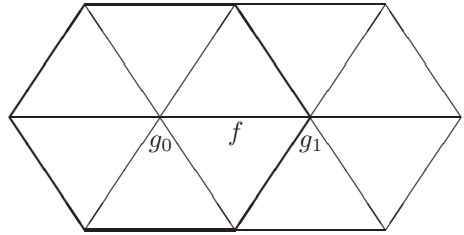
Fig. 3: The extended macroelement $\Omega_f^e$ corresponding to the
union of the
two macroelements $\Omega_{g_0}$
(outlined by the thick lines)
and $\Omega_{g_1}$, $n=2$.
\end{center}

\vskip.75in

\setlength{\unitlength}{0.5cm}
\begin{center}
\begin{picture}(0,0)
\put(0,0){\line(2,3){2}}
\put(0,0){\line(-2,3){2}}
\put(0,0){\line(-2,-3){2}}
\put(0,0){\line(2,-3){2}}
\put(-4,0){\line(1,0){8}}
\put(4,0){\line(1,0){4}}
\put(4,0){\line(2,-3){2}}
\put(4,0){\line(2,3){2}}
\put(2,3){\line(1,0){4}}
\put(2,-3){\line(1,0){4}}
\put(8,0){\line(-2,3){2}}
\put(8,0){\line(-2,-3){2}}
\put(1.8,-.6){$f$}
\put(-.3,-.8){$g_0$}
\put(3.7,-.8){$g_1$}
\thicklines
\put(-4,0){\line(2,3){2}}
\put(-4,0){\line(2,-3){2}}
\put(4,0){\line(-2,3){2}}
\put(4,0){\line(-2,-3){2}}
\put(-2,3){\line(1,0){4}}
\put(-2,-3){\line(1,0){4}}
\end{picture}
\end{center}

\vskip.75in
In the special case that $\dim f = 0$, i.e., $f$ is a vertex, then
$\Omega_f^{\eps} = \Omega_f$. In general, if $f,g \in \Delta(\T_h)$
with $g \in \Delta(f)$ then
\[
\Omega_f \subset \Omega_g \quad \text{and } \Omega_g^e \subset
\Omega_f^e.
\]
We shall assume throughout that all the 
macroelements of the form $\Omega_f$ and $\Omega_f^e$, for 
$f \in \Delta(\T_h)$,  are contractive.
We let $\T_{f,h}$ denote the restriction of $\T_h$ to $\Omega_f$,
while $\T_{f,h}^e$ is the corresponding restriction of $\T_h$ to $\Omega_f^e$.
It is straightforward to check that a consequence of the
shape regularity of the family
$\{\T_h \}$ is that the ratio $|\Omega_f^e|/|\Omega_f|$ is uniformly
bounded. Furthermore, the coverings $\{\Omega_f \}_{f \in
  \Delta(\T_h)}$ and  $\{\Omega_f^e \}_{f \in
  \Delta(\T_h)}$ of the domain $\Omega$ both have the bounded
overlap property, i.e., the sum of the characteristic functions is
bounded uniformly in $h$.

\subsection{Construction of the projection for scalar valued functions}
\label{ord-fcns}
To motivate the construction for the general case of $k$ forms given
below, we will first give an outline of how the projection is
constructed for zero forms, i.e. for scalar valued functions.  The
projection $\pi_h^0$ will map the space $H^1(\Omega) =
H\Lambda^0(\Omega)$ into $\P_r\Lambda^0(\T_h)$, the space of
continuous piecewise polynomials of degree $r$ with respect to the
partition $\T_h$.  The space $\P_r\Lambda^0(\T_{f,h})$ is the
restriction of the space $\P_r\Lambda^0(\T_h)$ to $\T_{f,h}$, and
$\0\P_r\Lambda^0(\T_{f,h})$ is the subspace of
$\P_r\Lambda^0(\T_{f,h})$ of functions which vanish on the boundary of
$\Omega_f$, $\partial \Omega_f$.  Of course, by the zero extension the
space $\0\P_r\Lambda^0(\T_{f,h})$ can also be considered as a subspace
of $\P_r\Lambda^0(\T_h)$.

A key tool for the construction is the local projection
$P_f^0 : H^1(\Omega_f) \to \P_r(\T_{f,h})$, associated to each  $f \in
\Delta(\T_h)$.
If $\dim f = 0$, such that $f$ is a vertex, we define $P_f^0$ by
$P_f^0 u \in \P_r\Lambda^0(\T_{f,h})$ as the $H^1$ projection of $u$, i.e.,
$P_f^0 u$ is the solution of:
\begin{align*}
\langle P_f^0 u, 1 \rangle_{\Omega_f} &= \langle u, 1 \rangle_{\Omega_f},
\\
\langle d P_f^0 u, d v \rangle_{\Omega_f} &=
\langle d u, d v \rangle_{\Omega_f}, \quad v \in \P_r(\T_{f,h}).
\end{align*}
Of course, for zero forms the exterior derivative, $d$, can be
identified with the ordinary gradient operator.
When $1 \le \dim f \le n$, we first define the space
\[
\breve\P_r\Lambda^0(\T_{f,h}) = \{u \in \P_r\Lambda^0(\T_{f,h})\, |\,  
\tr_f u \in \0\P_r(f) \,
\}.
\]
We then define $P_f^0 u \in \breve \P_r\Lambda^0(\T_{f,h})$ as the solution of
\begin{equation*}
\langle d P_f^0 u, d v \rangle_{\Omega_f} =
\langle d u, d v \rangle_{\Omega_f}, \quad v \in \breve \P_r\Lambda^0(\T_{f,h}).
\end{equation*}
The projection $\pi_h^0$ will be defined recursively with respect to
the dimensions of the subsimplices of the triangulation $\T_h$.
More precisely, we will utilize a sequence of local operators
$\{\pi_{m,h}^0\}_{m=0}^n$, and define $\pi_h^0 = \pi_{n,h}^0$.
Dropping the dependence on $h$, the operators $\pi_m^0$ are
defined recursively by
\begin{equation}\label{pi-m-0-recursion}
\pi_{m}^0 u = \pi_{m-1}^0 u + \sum_{f \in \Delta_m(\T_h)} E_{f}^0 \tr_f
P_{f}^0(u -  \pi_{m-1}^0 u), \quad 1 \le m  \le n.
\end{equation}
Here $E_f^0 : \0\P_r(f) \to \0\P_r\Lambda^0(\T_{f,h}) 
\subset \P_r\Lambda^0(\T_h)$ is
the harmonic extension operator determined by
\[
\langle d E_f^0 \phi, d v \rangle_{\Omega_f} = 0,
\quad v \in \0\P_r\Lambda^0(\T_{f,h}), \tr_f v = 0,
\]
and that $\tr_f E_f^0$ is the identity on $\0\P_r(f)$. We observe that
the dependency of the operator $E_f^0$ on the degree $r$ is
suppressed. It is a key property that $\tr_g E_f^0 \phi = 0$
for all $g \in \Delta(\T_h)$, $\dim g \le \dim f$, and $g \neq f$.
For the vertex degrees of freedom we will use an alternative extension
operator. We simply define $\pi_{0,h}^0 = \pi_0^0$ by
\[
\pi_0^0 u = \sum_{f \in \Delta_0(\T_h)} \E_{f}^0 \tr_f P_{f}^0 u
= \sum_{f \in \Delta_0(\T_h)} \E_{f}^0 (P_f^0 u)(f)
\]
where, for any $\alpha \in \R$, $\E_f^0\alpha$ is the piecewise linear
function with value $\alpha$ at the vertex $f$ and value zero at all
other vertices. Hence, for $f \in \Delta_0(\T_h)$ we have $\E_f^0 =
E_f^0$ if $r =1$. The reason for choosing the special low order extension
operator for vertices is not essential at this point, but will be
needed later to make sure that the projections $\pi_h^k$
commute with the exterior derivative.

The key result for the construction above is that the operator
$\pi_h^0$ is a projection.

\begin{lem}
The operator $\pi_h^0$ is a projection onto $\P_r\Lambda^0(\T_h)$.
\end{lem}
\begin{proof}
To see that $\pi^0= \pi_h^0$ is a projection, we only need to check 
that if $u \in
\P_r\Lambda^0(\T_h)$, then for all $f \in \Delta(\T_h)$, $\tr_f \pi^0 u = \tr_f
u$.  We do this by induction on $m$, where $m$ corresponds to the
dimension of the face $f \in \Delta(\T_h)$.
We assume throughout that $u \in \P_r\Lambda^0(\T_h)$.
We will show that the operator $\pi_m^0$ has the property that
\begin{equation}\label{prop-pi_m^0}
\tr_f \pi_m^0 u = \tr_f u \quad \text{if } f \in \Delta(\T_h) \quad
\text{with } \dim f
\le m,
\end{equation}
and since $\pi^0 = \pi_n^0$ this will establish the desired result.
If $f \in \Delta_0(\T_h)$
then $P_h^0u = u|_{\Omega_f}$. By construction, it therefore follows
that \eqref{prop-pi_m^0} holds for $m=0$.
Assume next that \eqref{prop-pi_m^0} holds for $m -1$, where $1 \le m \le
n$. It follows that for any $f \in \Delta_{m}(\T_h)$
we have $\tr_f(u - \pi_{m-1}^0 u) \in \0\P_r(f)$, and therefore
$P_f^0(u - \pi_{m-1}^0 u) = u - \pi_{m-1}^0 u$.
It follows by construction that
$\tr_g \pi_m^0u = \tr_g \pi_{m-1}^0u = \tr_g u$ for $g \in \Delta(\T_h)$, with
$\dim g < m$, while  for $f \in \Delta_m(\T_h)$ we have
\[
\tr_f \pi_m^0 u = \tr_f (\pi_{m-1}^0 u + P_f^0(u - \pi_{m-1}^0 u) =
\tr_f u.
\]
Therefore, \eqref{prop-pi_m^0} holds for $m$ and the proof is completed.
\end{proof}

It follows from the construction above that the operator $\pi_h^0$
is local. For example, for any $T \in \T_h$ we have that
$(\pi_{0,h}^0u)_T$ depends only on $u$ restricted to the extended
macroelement $\Omega_T^e$. Define 
$D_{m,T} \subset \Omega$
by 
\begin{equation}\label{D-recursion}
D_{m,T} = \cup \{\, D_{m-1,T'} \, |\, T' \in \T_{f,h}, \, f \in \Delta_m(T) \,
\}, \quad D_{0,T} = \Omega_T^e.
\end{equation}
It follows from \eqref{pi-m-0-recursion}
that $(\pi_{m,h}u)|_T$ depends only on $u|_{D_{m,T}}$. In particular,
$(\pi_h^0u)|_T$
depends only on $u|_{D_T}$, where $D_T = D_{n,T}$.

The operator $\pi_h^0$ satisfies the following local estimate. 

\begin{thm}
\label{bound0}
Let $T \in \T_h$.
The operator $\pi_h^0$ satisfies the bounds
\[
\|\pi_h^0 u\|_{L^2(T)} \le C (\| u \|_{L^2(D_T)} +
h_T \| d u \|_{L^2(D_T)})
\]
and
\[
\|d \pi_h^0 u\|_{L^2(T)} \le C \| d u \|_{L^2(D_T)},
\]
where the constant $C$ is independent of $h$ and $T \in \T_h$.
\end{thm}

In fact, this result is just a special case of Theorem~\ref{bound}
below, so we omit the proof here. Of course, due to the bounded
overlap property of the covering $\{D_T\}_{T \in \T_h}$ of $\Omega$,
derived from the corresponding property of $\{\Omega_f^e\}$,
global estimates follow directly from the local estimates above.

\subsection{Differential forms and finite element spaces}
\label{sec:notation-prelim}
We will basically adopt the notation from \cite{bulletin}.
The spaces $\P_r\Lambda^k(\T_h) \subset H\Lambda^k(\Omega)$
can be characterized as the space of piecewise polynomial $k$ forms
$u$ of degree less than or equal to $r$, such that the trace, 
$\tr_f u$, is continuous
for all $f \in \Delta(\T_h)$, with $\dim f \ge k$,
where we recall that the trace, $tr_f$, of a differential form is
defined by restricting to $f$ and applying the form only to tangent vectors.
The space $\P_r^-\Lambda^k(\T_h) \subset H\Lambda^k(\Omega)$
is defined similarly, but on each element $T \in \T_h$,
$u$ is restricted  to be in $\P_r^-\Lambda^k \subset \P_r\Lambda^k$.
Here, the polynomial class $\P_r^-\Lambda^k$ consists of all
elements $u$ of $\P_r\Lambda^k$ such that $u$ contracted with the
position vector $x$, $u \lrcorner x$, is in $\P_r\Lambda^{k-1}$.
Hence, for each $k$ we have a sequence of nested spaces
\[
\P_1^-\Lambda^k(\T_h) \subset \P_1\Lambda^k(\T_h) 
\subset \P_2^-\Lambda^k(\T_h)\subset \ldots H\Lambda^k(\Omega).
\]
In particular, $\P_r^-\Lambda^0(\T_h) = \P_r\Lambda^0(\T_h)$, and 
$\P_r^-\Lambda^n(\T_h) = \P_{r-1}\Lambda^n(\T_h)$.

Instead of distinguishing the theory for the spaces
$\P_r^-\Lambda^k(\T_h)$ and $\P_r\Lambda^k(\T_h)$ we will use
the simplified notation $\P\Lambda^k(\T_h)$ to denote either a space
of the family
$\P^-_r\Lambda^k(\T_h)$ or $\P_r\Lambda^k(\T_h)$. More precisely, 
we assume that we are given a sequence of spaces  
$\P\Lambda^k(\T_h)$, for $k= 0, 1, \ldots , n$, such that the
corresponding polynomial sequence $(\P\Lambda, d)$, given by
\begin{equation}\label{pol-complex}
\begin{CD}
\R \to  \P\Lambda^0(\R^n) @>d>> \P\Lambda^{1}(\R^n)
 @>d>> \cdots @>d>> \P\Lambda^n(\R^n)\to 0
\end{CD}
\end{equation}
is
an exact complex (cf. Section 5.1.4 of \cite{bulletin}).
In particular, this allows for combinations of spaces taken from the 
two families
$\P_r^-\Lambda^k(\T_h)$ and $\P_r\Lambda^k(\T_h)$.
For any $f \in \Delta(\T)$, with $\dim f \ge k$, the space 
$\P\Lambda^k(f) = \tr_f\P\Lambda^k(\T_h)$, while $\0\P\Lambda^k(f) = 
\{v \in \P \Lambda^k(f)\, |\,  \tr_{\partial f} v =0\}$.
The corresponding polynomial complexes of the form  
$(\P\Lambda(f),d)$ are all exact. Furthermore, the complexes with homogeneous
boundary conditions, $(\0\P\Lambda(f),d)$, given by 
\begin{equation}\label{pol-complex-0} 
\begin{CD}
\0\P\Lambda^0(f) @>d>> \0\P\Lambda^{1}(f)
 @>d>> \cdots @>d>> \0\P\Lambda^{\dim f}(f)\to \R
\end{CD}
\end{equation}
are also exact. 

We recall that the spaces $\P\Lambda^k(\T_h)$ admit degrees of freedom
of the form
\begin{equation}\label{DOF}
\int_f \tr_f u \wedge \eta, \quad \eta
\in \P'(f,k), \quad f \in \Delta(\T_h),
\end{equation}
where $\P'(f,k) \subset \Lambda^{dim f - k}(f)$ is a polynomial space
of differential forms and the symbol $\wedge$ is used to denote 
the exterior product. These degrees of freedom
uniquely determine an element in $\P\Lambda^k(\T_h)$,
(cf. Theorem 5.5 of \cite{bulletin}). In fact, if
\[
\P\Lambda^k(\T_h) = \P_r^-\Lambda^k(\T_h), \quad \text{then }
\P'(f,k) = \P_{r+k-\dim f -1}\Lambda^{\dim f -k}(f),
\]
while if
\[
\P\Lambda^k(\T_h) = \P_r\Lambda^k(\T_h), \quad \text{then }
\P'(f,k) = \P_{r+k-\dim f}^-\Lambda^{\dim f -k}(f).
\]
If $v \in \0\P\Lambda^k(f)$,
then $v$ is uniquely determined by the functionals derived from $\P'(f,k)$.
Furthermore, any $v \in \P\Lambda^k(f)$ is uniquely determined by $\P'(g,k)$
for all $g \in \Delta(f)$.
In particular, if $\dim f < k$
then $\P'(f,k)$ is empty, while $\P'(f,k)$ is always nonempty
if $\dim f = k$. For $\dim f >k$ the set $\P'(f,k)$ can also be empty
if the polynomial degree $r$ is sufficiently low. 

The local spaces $\P\Lambda^k(\T_{f,h})$ and $\P\Lambda^k(\T_{f,h}^e)$ are 
defined by restricting the space $\P\Lambda^k(\T_{h})$ to the
macroelements $\Omega_f$ or $\Omega_f^e$.
It follows from 
the assumption that $\Omega_f$ and $\Omega_f^e$ are contractive, that
all the local complexes $(\P\Lambda(\T_{f,h}),d)$
and $(\P\Lambda(\T_{f,h}^e),d)$ are exact. 
The same holds for the subcomplexes 
$(\0\P\Lambda(\T_{f,h}),d)$ and
$(\0\P\Lambda(\T_{f,h}^e),d)$, corresponding to the subspaces of
functions with zero trace on the
boundary of the macroelements.

For a given triangulation $\T_h$, the spaces of 
lowest order polynomial degree, $\P_1^-\Lambda^k(\T_h)$,
i.e., the space of Whitney forms, 
will play a special role in our construction.
The dimension of this space is equal to the number of elements in
$\Delta_k(\T_h)$, and the properties of these spaces will
in some sense reflect the properties of the triangulation.
Therefore, this space will be used to transfer information between
different macroelements, cf. Section~\ref{sec:whitney} below.
For $k=0$ this space is just $\P_1\Lambda^k(\T_h)$, 
the space of continuous piecewise
linear functions. The natural basis for this space is the set of
generalized barycentric coordinates,
defined to be one at one vertex, and zero at all other vertices.
It follows from the discussion above that the degrees of freedom for
the space $\P_1^-\Lambda^k(\T_h)$, $0 \le k \le n$, are $\int_fu$ for all
$f \in \Delta_k(\T_h)$.
In fact, if $f =[x_0,x_1, \ldots x_k] \in \Delta_k(\T_h)$, we define the
Whitney form associated to $f$, $\phi_f^k \in \P_1^-\Lambda^k(\T_h)$, by
\[
\phi_f^k = \sum_{i=0}^k (-1)^i \lambda_i d \lambda_0 \wedge \cdots
\wedge \widehat{d \lambda_i} \wedge \cdots \wedge d \lambda_k,
\]
where $\lambda_0, \lambda_1, \ldots, \lambda_k$ are the barycentric coordinates
associated to the vertices $x_i$.
The basis function $\phi_f^k$ reduces to a constant $k$ form on $f$, i.e.,
$\tr_f \phi_f^k \in \P_0\Lambda^k(f)$, and it has the property that
$\tr_g \phi_f^k = 0$ for $g \in \Delta_k(\T_h)$, $g\neq f$.
In fact, if $\vol_f \in \P_0\Lambda^k(f)$ is the volume form on $f$, 
scaled such that $\int_f \vol_f  = 1$,
then
\[
\tr_f \phi_f^k = (k!)^{-1}\vol_f,
\]
cf. \cite[Section 4.1]{acta}. Furthermore, the map 
$\vol_f \to \E_f^k \vol_f = k! \phi_f^k$
defines an extension operator 
$\E_f^k : \P_0\Lambda^k(f) \to \0\P_1^-\Lambda^k(\T_{f,h})$
for any $f \in \Delta_k(\T_h)$.
We observe that the operators $\E_f^k$ are natural generalizations of the 
piecewise linear extension operators
$\E_f^0$, introduced above for scalar valued functions.
In fact, any element $u$ of $\P_1^-\Lambda^k(\T_h)$ admits the
representation
\begin{equation}\label{Whitney-rep}
u = \sum_{f \in \Delta_k(\T_h)} \left(\int_f \ \tr_f u\right) \E_f^k\vol_f.
\end{equation}

We finally note that it follows from Stokes' theorem that if
$f = [x_0, x_1, \ldots ,x_{k+1}]$, 
and $u$ is a sufficiently smooth $k$ form on $f$, then
\begin{equation}\label{stokes}
\int_f du = \sum_{j= 0}^{k+1}(-1)^j\int_{f_j} \tr_{f_j}u,
\end{equation}
where $f_j = [x_0, \ldots, x_{j-1}, \hat x_j, x_{j+1}, \cdots x_{k+1}]$.
Here the factor $(-1)^j$ enters as a consequence of 
orientation.

\section{A special projection onto the Whitney forms}
\label{sec:whitney}
Recall that the purpose of this paper is to construct local cochain
projections $\pi_h^k$ which map $H\Lambda^k(\Omega)$ boundedly onto
the piecewise polynomial space $\P\Lambda^k(\T_h)$. Furthermore, in
the construction of $\pi_h^0$ given above, the construction of $\tr_f
\circ \pi_h^0$ is based on a local projection, $P_f^0$, defined with
respect to the associated macroelement $\Omega_f$.  Therefore one
might hope that all the projections $\pi_h^k$ have the property that
$\tr_f \circ \pi_h^k$ is defined from a local projection operator
defined on $\Omega_f$ for $f \in \Delta(\T_h)$, $\dim f \ge k$.
However, a simple computation in two space dimensions, and with
$\P\Lambda^k(\T_h) = \P_1^-\Lambda^k(\T_h)$, will convince the reader
that if $ f = [x_0,x_1] \in \Delta_1(\T_h)$, then
\[
\int_f \tr_f d \pi_h^0 u = \int_{x_0}^{x_1} \frac{d}{ds} \pi_h^0 u \, ds
= (\pi_h^0 u)(x_1) - (\pi_h^0 u)(x_0),
\]
and the right hand side here clearly depends on $u$ restricted to 
the union of the macroelements associated to the 
vertices $x_0$ and $x_1$. 
Therefore, $\int_f \tr_f \pi_1^h du = \int_f \tr_f d \pi_h^0 u$ 
must also depend on $u$ restricted to 
the union of these macroelements, and this domain is exactly equal to the 
extended macroelement $\Omega_f^e$.
This motivates why the extended macroelements, 
$\Omega_f ^e$, for $f \in \Delta_k(\T_h)$, will appear in the
construction below.
In fact, a special projection operator, 
$R_h^k : H \Lambda^k(\Omega) \to \P_1^- \Lambda^k(\T_h) 
\subset \P\Lambda^k(\T_h)$,
will be utilized in the construction of $\pi_h^k$ to make sure that
\[
\int_f \tr_f \pi_h^k du = \int_f \tr_f d\pi_h^{k-1} u  
= \int_{\partial f} \tr_{\partial f} \pi_h^{k-1}u,
\]
for all $f \in \Delta_k(\T_h)$.

The operator $R_h^k$ will commute with the exterior derivative, and it
is a projection onto $\P_1^-\Lambda^k(\T_h)$. Therefore, in the case
of lowest polynomial degree, when $\P\Lambda^k(\T_h) =
\P_1^-\Lambda^k(\T_h)$, we will take $\pi_h^k = R_h^k$.  However,
another key property of the operator $R_h^k$ is that in the general
case, when $\P_1^-\Lambda^k(\T_h)$ is only contained in
$\P\Lambda^k(\T_h)$, we will have
\begin{equation}\label{R-key-prop}
\int_f \tr_f R_h^k u = \int_f \tr_f u, \quad f \in \Delta_k(\T_h), \, u 
\in \P\Lambda^k(\T_h),
\end{equation}
i.e., the operator $R_h^k$ preserves the mean values of the traces of function 
in $\P\Lambda^k(\T_h)$
on subsimplexes $f$ of dimension $k$. The rest of this section is devoted to 
the construction of the operator $R_h^k$, and the derivation of the 
key properties given in Theorem~\ref{whitney-commute} below.

\subsection{Tools for the construction}
Throughout this section the dependence on the mesh parameter $h$ is
suppressed in order to simplify the notation.
To define the  special projection $R^k$ onto the Whitney forms,
$\P_1^-\Lambda^k(\T)$, 
we will use local projections,
$Q_{f}^k$, defined with respect to the extended macroelements
$\Omega_f^{\eps}$.
We define the projection $Q_{f}^k :
H\Lambda^k(\Omega_f^{\eps}) \to \P \Lambda^k(\T_{f}^{\eps})$ by the system
\begin{align*}
\<Q_{f}^k u, d\tau\>_{\Omega_f^{\eps}} &= \<u,d\tau\>_{\Omega_f^{\eps}},
\quad \tau \in
\P \Lambda^{k-1}(\T_{f}^{\eps}),\\
\<d Q_{f}^k u,dv\>_{\Omega_f^{\eps}} &= \<du,dv \>_{\Omega_f^{\eps}},
\quad v \in \P \Lambda^{k}(\T_{f}^\eps).
\end{align*}
For $k=0$, the first equation should be replaced by
a mean value condition, so that $Q_{f}^0 = P_{f}^0$.
This system has a unique solution due to the exactness of the complex
$(\P \Lambda(\T_{f}^{\eps}),d)$. Furthermore, by construction we have
\begin{equation}\label{d-Q-commute}
Q_{f}^k du = dQ_{f}^{k-1}u, \quad 0 < k \le  n.
\end{equation}
We will also find it useful to introduce the 
operator
$Q_{f,-}^{k}: H\Lambda^{k}(\Omega_f^e) \to
\P \Lambda^{k-1}(\T_{f}^e)$ defined by the corresponding reduced  system
\begin{align*}
\<Q_{f,-}^{k} u, d\tau\>_{\Omega_f^{\eps}} &= 0,
\quad \tau \in
\P \Lambda^{k-2}(\T_{f,h}^{\eps}),
\\
\<d Q_{f,-}^{k} u,dv\>_{\Omega_f^{\eps}}
&= \<u,dv \>_{\Omega_f^{\eps}},
\quad v \in \P \Lambda^{k-1}(\T_{f,h}^\eps).
\end{align*}
As a consequence, the projection $Q_f^k$ can be expressed as 
\begin{equation}\label{Q-decomp}
Q_f^k = d Q_{f,-}^k + Q_{f,-}^{k+1}d.
\end{equation}
To make this relation true also in the case when $k=0$ and $f \in
\Delta_0(\T)$,
the operator $dQ_{f,-}^0$ should have the interpretation that 
$dQ_{f,-}^0u$ is the constant $\int_{\Omega_f}u \wedge
\vol_{\Omega_f}$ on $\Omega_f$, where $\vol_{\Omega_f}$ is the volume
form on $\Omega$, restricted to $\Omega_f$ and scaled such that $\int_{\Omega_f}\vol_{\Omega_f} =1$.

To motivate the rest of the tools we need to for our construction,
consider again $d\pi^0u$ in the special case when $\P\Lambda^k(\T) =
\P_1^-\Lambda^k(\T)$.
To obtain a commuting relation of the form $d\pi^0u =
\pi^1du$, we have to be able to express $d\pi^0
u$
in terms of $du$. However, using the notation just introduced, we have 
\[
d\pi^0 u = \sum_{g \in \Delta_0(\T)}\Big[\Big(\int_{\Omega_g}u \wedge
\vol_{\Omega_g}\Big) + \tr_g(Q_{g,-}^1du)\Big]d\E_g\vol_g.
\]
The second part of this sum is already expressed in terms of $du$.
By combining the contributions from neighbouring macroelements 
we wil see that also the first part of the right hand side can
be expressed in term of $du$.
If $f = [x_0,x_1] \in
\Delta_1(\T)$,
we have
\[
\int_f \tr_f \sum_{g \in \Delta_0(\T)}\Big(\int_{\Omega_g}u \wedge
\vol_{\Omega_g}\Big)d\E_g\vol_g = \int_{\Omega_f^e}u \wedge
(\vol_{\Omega_{g_1}} - \vol_{\Omega_{g_0}}),
\]
where $g_i = [x_i]$.  Furthermore, $\vol_{\Omega_{g_1}} -
\vol_{\Omega_{g_0}} \in \P_0\Lambda^n(\T_f^e)=
\P_1^-\Lambda^n(\T_f^e)$, and with vanishing
integral.
As a consequence, there exists $z_f^1 \in
\0\P_1^-\Lambda^{n-1}(\T_f^e)$ such that $dz_f^1 = \vol_{\Omega_{g_0}} -
\vol_{\Omega_{g_1}}$, and by integration by parts 
\[
\int_f \tr_f \sum_{g \in \Delta_0(\T)}\Big(\int_{\Omega_g}u \wedge
\vol_{\Omega_g}\Big)d\E_g\vol_g = -\int_{\Omega_f^e}u
\wedge dz_f^1 = \int_{\Omega_f^e}du
\wedge z_f^1.
\]
By utilizing the representation \eqref{Whitney-rep}, we therefore
obtain 
\[
\sum_{g \in \Delta_0(\T)}\Big(\int_{\Omega_g}u \wedge
\vol_{\Omega_g}\Big)d\E_g\vol_g = \sum_{f \in \Delta_1(\T_h)} \Big(\int_{\Omega_f^e}du
\wedge z_f^1   \Big) \E_f^k\vol_f.
\]
This discussion shows that to construct 
local cochain  projections, we must utilize relations between local
operators defined on different macroelements.
To derive the proper relations, we introduce an operator 
\[
\delta : \bigoplus_{g \in \Delta_m(\T)}\0\P_1^-\Lambda^k(\T_g^e) \to 
\bigoplus_{f \in \Delta_{m+1}(\T)}\0\P_1^-\Lambda^k(\T_f^e).
\]
If $f = [x_0, \ldots, x_{m+1}] \in \Delta_{m+1}(\T)$, then the 
component $(\delta u)_f$ of $\delta u$ is defined by
\begin{equation*}
(\delta u)_f 
= \sum_{j=0}^{m+1} (-1)^j u_{f_j},
\end{equation*}
where, as above, 
$f_j = [x_0, \ldots, x_{j-1}, \hat x_j, x_{j+1}, \cdots x_{m+1}]$, and
$u_{f_j}$
the corresponding component of $u$. 
We will also consider the exterior derivative $d$ as an operator
mapping $\bigoplus_{g \in \Delta_m(\T)}\0\P_1^-\Lambda^k(\T_g^e)$ to 
$\bigoplus_{g \in \Delta_m(\T)}\0\P_1^-\Lambda^{k+1}(\T_g^e)$ by
applying it to each component. Hence, the two operators 
$d \circ \delta$ and $\delta \circ d$
both map $\bigoplus_{g \in \Delta_m(\T)}\0\P_1^-\Lambda^k(\T_g^e)$ into
$\bigoplus_{f \in \Delta_{m+1}(\T)}\0\P_1^-\Lambda^{k+1}(\T_f^e)$.
In fact, we have the stucture of a double complex which resembles 
the well--known \v{C}ech--de Rham complex, cf. \cite{Bott-Tu}.
The following two properties of the operator
$\delta$ are crucial.

\begin{lem}
\label{delta-prop}
\begin{equation*}
d\circ \delta  = \delta \circ d , \quad \text{and} \quad
\delta \circ \delta  =0.
\end{equation*}
\end{lem}
\begin{proof}
It follows directly from the definition of $\delta$ that for 
$f = [x_0, \ldots, x_{m+1}] \in \Delta_{m+1}(\T)$
\begin{equation*}
(d\circ \delta u)_f = (\delta \circ d u)_f = 
\sum_{j=0}^{m+1} (-1)^j du_{f_j}. 
\end{equation*}
If we further denote by $f_{ij}$ the subsimplex
of $f$ obtained by deleting both $x_i$ and $x_j$, then 
\begin{align*}
(\delta \circ \delta u)_f &= 
 \sum_{j=0}^{m+1} (-1)^j (\delta u)_{f_j}\\
&=\sum_{j=0}^{m+1} (-1)^j
\Big[\sum_{i=0}^{j-1} (-1)^i u_{f_{ij}}
- \sum_{i=j+1}^{m+1} (-1)^i u_{f_{ij}} \Big] =0,
\end{align*}
since for each $i,j = 0, \ldots m+1$, with $i \neq j$, the term
$u_{f_{ij}}$ 
appears exactly twice with opposite signs.
\end{proof}

The construction of the projection $R^k = R_h^k$ will depend on local
weight functions, $z_f^k \in \0\P_1^-\Lambda^{n-k}(\T_f^e)$ for $f \in
\Delta_k(\T)$.
In particular, the function $z_f^0 \in \P_0\Lambda^n(\T_f)$ for $f
\in \Delta_0(\T)$ will be given by $z_f^0 =
\vol_{\Omega_f}$.
For $k= 1, 2, \ldots ,n$, the functions $z_f^k \in
\0\P_1^-\Lambda^{n-k}(\T_f^e)$ 
are defined recursively to satisfy the conditions
\begin{equation}
\label{induction-hyp2}
d z_f^k = (-1)^k (\delta z^{k-1})_f,
\end{equation}
and 
\begin{equation}
\label{addconds}
\< z_f^k, d \tau \>_{\Omega_f^e} = 0, \qquad \tau \in \0 \P_1^-
\Lambda^{n-k-1}(\T_f^e),
\end{equation}
for any $f \in \Delta_k(\T)$.
We will not give an explicit construction of the functions $z_f^k$.
However, we have the following basic result.

\begin{lem}
\label{well-defined}
The weight functions $z_f^k \in \0\P_1^-\Lambda^{n-k}(\T_f^e)$
exist and are uniquely determined by $z_f^0$
and the conditions \eqref{induction-hyp2} and  \eqref{addconds}.
\end{lem}
\begin{proof}
We establish the existence of the functions $z_f^k$
by induction on $k$.
Let $f = [x_0,x_1] \in \Delta_1(\T)$. Then
\[
(\delta z^0)_f = (z_{f_1}^0 - z_{f_0}^0)
= \vol_{\Omega_{f_1}} - \vol_{\Omega_{f_0}},
\]
which implies that
$\int_{\Omega_f^e} (\delta z^0)_f = 0$.
Hence, by the exactness of the complex
\begin{equation*}
\0\P_1^-\Lambda^{n-1}(\T_f^e) \xrightarrow {d} \P_0\Lambda^n(\T_f^e) \to \R,
\end{equation*}
there exists $z^1_f \in \0\P_1^-\Lambda^{n-1}(\T_f^e)$ satisfying
$d z^1_f = - (\delta z^0)_f$.
Next, assume we have constructed $z^{k-1} \in 
\bigoplus_{f \in \Delta_{k-1}(\T)}\0\P_1^-\Lambda^{n-k+1}(\T_f^e)$ such that
$d z_f^{k-1} = (-1)^{k-1}(\delta z^{k-2})_f$ for all $f \in
\Delta_{k-1}(\T)$. From 
Lemma~\ref{delta-prop}
we obtain 
\begin{equation*}
(d \circ \delta) z^{k-1} = (\delta \circ d) z^{k-1}
= (-1)^{k-1}(\delta \circ \delta) z^{k-2} = 0,
\end{equation*}
and for each $f \in \Delta_k(\T)$ the complex $(d,
\0\P_1^-\Lambda(\T_f^e))$ is exact. Therefore, we can conclude that 
there is a $z_f^k\in
\0\P_1^-\Lambda^{n-k}(\Delta_k(\T_f^e))$ such that 
\eqref{induction-hyp2} holds. This completes the induction argument.
Finally, we observe that it is a consequence of \eqref{addconds}
and the exactness of
the complex 
$(d,\0\P_1^-\Lambda(\T_f^e))$ that each function $z_f^k$ is uniquely determined.
\end{proof}
We will use the functions $z_f^k$ to define an operator $M^k : = M_h^k : H\Lambda^k(\Omega) \to \P_1^-
\Lambda^k(\T)$ by 
\[
M^k u = \sum_{f \in \Delta_{k}(\T)} \Big(\int_{\Omega_f^e} u
\wedge z_f^{k}\Big) \E_f^{k}\vol_f.
\]
Note that $M^k u$ is a generalization for $k$-forms of the expression
\[
 \sum_{f \in \Delta_{0}(\T)} \Big(\int_{\Omega_f} u
\wedge \vol_{\Omega_f}\Big) \E_f \vol_f.
\]
appearing above in the case of zero-forms.
It follows from the construction of the functions $z_f^k$
that the operator $M^k$ commutes with the exterior derivative.

\begin{lem}
\label{zfk-ident}
For any $v \in H\Lambda^{k-1}(\Omega)$ the identity $dM^{k-1}v = M^kdv$
holds.
\end{lem}
\begin{proof}
We have to show that 
\begin{equation*}
\sum_{g \in \Delta_{k-1}(\T)}  \Big(\int_{\Omega_g^e}
v \wedge z_g^{k-1}\Big)d\E_g^{k-1} \vol_g
=  \sum_{f \in \Delta_{k}(\T)} \Big(\int_{\Omega_f^e}
dv \wedge z_f^{k}\Big)\E_f^{k} \vol_f
\end{equation*}
for any $v \in H\Lambda^{k-1}(\Omega)$.
Since both sides of this equation are elements of $\P_1^- \Lambda^k(\T)$, we
need only check that the integrals of
their traces 
are the same over each $f = [x_0.x_1, \ldots ,x_k] \in \Delta_k(\T)$.
Now it follows from the properties of the extension operators $\E_f^k$
that the integral of the right hand side is simply
$\int_{\Omega_f^e} dv \wedge z_f^k$, while \eqref{stokes} implies that
the corresponding
integral of the left hand side is
\begin{equation*}
\sum_{j=0}^k(-1)^j\int_{\Omega_{f_j}^e}v \wedge z_{f_j}^{k-1}
= \int_{\Omega_{f}^e}v \wedge (\delta z^{k-1})_f
= (-1)^k \int_{\Omega_{f}^e}v \wedge dz_f^{k},
\end{equation*}
where the last identity follows by \eqref{induction-hyp2}.
However, by integration by parts (cf. \cite[Section 2.2]{acta}), 
utilizing that $\tr_{\partial   \Omega_f^e}z_f^k = 0$, we have 
\[
\int_{\Omega_{f}^e}v \wedge dz_f^{k} = (-1)^k\int_{\Omega_{f}^e}dv
\wedge z_f^{k},
\]
and this completes the proof.
\end{proof}
We now define an operator $S^k = S_h^k : H\Lambda^k(\Omega) \to \P_1^-
\Lambda^k(\T)$ recursively by $S^0 = M^0$ and 
\[
S^{k}u = M^ku
+ \sum_{g \in \Delta_{k-1}(\T)}\Big(\int_g \tr_g [I - S^{k-1}]
Q_{g,-}^{k} u \Big) d \E_g^{k-1}\vol_g, \quad  1 \le k \le n.
\]
We recall that the operator $Q_{g,-}^{k}$ is a local operator with
range $\P\Lambda(\T_f^e)$.
However, by an inductive argument, it follows that the composition 
$\tr_f \circ S^k$
is a local operator mapping $H\Lambda^k(\Omega_f^e)$ into $\P_0\Lambda^k(f)$.
Therefore, the operators $S^k$ are indeed well defined.

A key property of the operator $S^k$ is the following result.
\begin{lem}\label{S-prop-mean}
For any $v \in \P\Lambda^{k-1}(\T)$ the identity 
\[
\int_f S^k dv = \int_f dv, \quad f \in \Delta_k(\T).
\]
holds.
\end{lem}
\begin{proof}
The proof goes by induction on $k$. For $k=0$ the space 
$d\P\Lambda^{k-1}(\T)$ should be interpreted as the space of constants
on $\Omega$, and since $S^0 = M^0$ reproduces constants
the desired identity holds.

Assume next that $k \ge 1$ and that the desired identity 
holds for $k-1$. By utilizing the result of Lemma~\ref{zfk-ident}
we obtain that the commutator $S^kd -dS^{k-1}$ has the representation
\[
S^kdv - dS^{k-1}v = \sum_{g \in \Delta_{k-1}(\T)}\Big(\int_g \tr_g [I - S^{k-1}]
Q_{g,-}^{k}d v \Big) d \E_g^{k-1}\vol_g, \quad v \in H\Lambda^{k-1}(\Omega).
\]
We recall that 
$Q_{g,-}^{k}d v = Q_g^{k-1} v - dQ_{g,-}^{k-1}v$,
and by the induction hypothesis
$\int_g \tr_g (I - S^{k-1})dQ_{g,-}^{k-1}v = 0$.
Therefore, for any $v \in H\Lambda^{k-1}(\Omega)$ the 
commutator above can be expressed as 
\begin{equation}\label{commutator}
S^kdv - dS^{k-1}v = \sum_{g \in \Delta_{k-1}(\T)}\Big(\int_g \tr_g [I - S^{k-1}]
Q_{g}^{k-1} v \Big) d \E_g^{k-1}\vol_g.
\end{equation}
However, since $Q_g^{k-1}$ is a projection onto $\P\Lambda^{k-1}(\T_g^e)$,
it follows from 
\eqref{Whitney-rep} that
\[
dS^{k-1}v =  \sum_{g \in \Delta_{k-1}(\T)}\Big(\int_g \tr_g S^{k-1}
Q_{g}^{k-1} v \Big) d \E_g^{k-1}\vol_g, \quad v \in
\P\Lambda^{k-1}(\T).
\]
By restricting to a function $v \in \P\Lambda^{k-1}(\T)$
equation  \eqref{commutator} therefore reduces to
\[
S^k dv = \sum_{g \in \Delta_{k-1}(\T)}\Big(\int_g \tr_g 
v \Big) d \E_g^{k-1}\vol_g.
\]
By integrating this representation over any 
$f = [x_0,x_1, \ldots ,x_k] \in \Delta_k(\T)$ we obtain
\[
\int_f S^k dv = \sum_{j=0}^k (-1)^j\int_{f_j} \tr_{f_j}v  = \int_f dv,
\]
where the final identity follows from \eqref{stokes}
This completes the induction argument.
\end{proof}

As a  direct consequence of the proof above, we have.

\begin{lem}\label{S-prop-com}
The identity \eqref{commutator} holds for $k= 1, 2, \ldots ,n$
and all $v \in H\Lambda^{k-1}(\Omega)$.
\end{lem}

\subsection{The projection $R_h^k$}
For each $k$, $0 \le k \le n$, the operator $R^k = R_h^k:
H\Lambda^k(\Omega) \to \P_1^- \Lambda^k(\T)$ is defined by
\[
R^k u = S^ku + \sum_{f \in \Delta_k(\T)}
\Big(\int_f \ \tr_f [I - S^k] Q_f^k u \Big) \E_f^k\vol_f.
\]
Recall that the operator $S^k$ is local in the sense that 
$\tr_f\circ S^k$ can be seen as a local operator mapping
$H\Lambda^k(\Omega_f^e)$
onto $\P\Lambda^k(f)$. It is immediate from this and the properties of
of the projection $Q_f^k$, 
that $\tr_f\circ R^k$ also is local.
In fact, for any $T \in \T$, $(R^ku)|_{T}$ only depends on $u|_{\Omega_T^e}$.
Furthermore, if $f \in \Delta_0(\T_h)$ then $Q_f^0 = P_f^0$.
Therefore, it follows that for $k= 0$ the operator $R_h^0$
is identical to the operator $\pi_{0,h}^0$, used in the construction of
the projection $\pi_h^0$ in Section~\ref{ord-fcns} above.

The key properties of the operator $R_h^k$ are given in the theorem below.

\begin{thm}
\label{whitney-commute}
The operators $R_h^k : H\Lambda^k(\Omega) \to \P_1^-\Lambda^k(\T_h)$ 
are cochain projections. Furthermore, they 
satisfy property \eqref{R-key-prop}, i.e., 
\[
\int_f \tr_f R_h^k u = \int_f \tr_f u, \quad f \in \Delta_k(\T_h), \,
u \in \P\Lambda^k(\T_h).
\]
\end{thm}

\begin{proof}
As above we use a notation where we suppress the dependence on $h$.
It is a consequence of the projection property of the operators
$Q_f^k$ that if 
$u \in \P\Lambda^k(\T)$ then 
\[
R^k u = \sum_{f \in \Delta_k(\T)}
\Big(\int_f \ \tr_f u \Big) \E_f^k\vol_f.
\]
However, this implies the identity \eqref{R-key-prop}, and an
immediate further consequence is that $R^k$ is a projection onto 
$\P_1^-\Lambda^k(\T_h)$. 

It remains to show that $R^k$ commutes with the exterior derivative. 
From the definition of $R^k$ and Lemma~\ref{S-prop-com} we have
\[
dR^ku = dS^{k}u + \sum_{f \in \Delta_{k}(\T)}\Big(\int_f \tr_f [I - S^{k}]
Q_{f}^{k} u \Big) d \E_f^{k}\vol_f = S^{k+1}du. 
\]
However, $S^{k+1}du = R^{k+1}du$ since
\begin{align*}
R^{k+1}du - S^{k+1}du &= \sum_{f \in \Delta_{k+1}(\T)}\Big(\int_f \tr_f [I - S^{k+1}]
Q_{f}^{k+1} du \Big)  \E_f^{k+1}\vol_f \\
&= \sum_{f \in \Delta_{k+1}(\T)}\Big(\int_f \tr_f [I - S^{k+1}]
dQ_{f,-}^{k+1} du \Big)  \E_f^{k+1}\vol_f = 0,
\end{align*}
where the last identity follows from Lemma~\ref{S-prop-mean}.
\end{proof}

The operators $R_h^k$ introduced above are local operators in the
sense that $(R_h^ku)|_T$ only depends on $u|_{\Omega_T^e}$ for any $T
\in \T_h$. Furthermore, for any fixed $h$ the operator $R_h^k$
is a bounded operator on $H\Lambda^k(\Omega)$.
The discussion of more precise local bounds
is delayed until the final section of the paper.

\section{Construction of the Projection: The General Case}
\label{general}
We finally turn to the construction of the projections $\pi_h^k$ in
the general case, in which $\P \Lambda^k(\T_h)$ denotes any
family of spaces of the form $\P^-_r\Lambda^k(\T_h)$ or $\P_r\Lambda^k(\T_h)$,
such that the corresponding polynomial sequence $(\P\Lambda^k, d)$, given by
\eqref{pol-complex} is an exact complex. In particular, the Whitney
forms, $\P_1^-\Lambda^k(\T_h)$, are a subset of $\P\Lambda^k(\T_h)$,
and in the special case when $\P\Lambda^k(\T_h) =
\P_1^-\Lambda^k(\T_h)$
we will take $\pi_h^k$ to be the operator $R_h^k$ constructed above.

In the construction we will utilize a decomposition  
of $\P\Lambda^k(\T_h)$ of the form
\begin{equation}\label{decomp}
\P\Lambda^k(\T_h) = \bigoplus_{f \in \Delta_k(\T_h)} \E_{f}^k(\P_0\Lambda^k(f))
+ \bigoplus_{\stackrel{f \in
  \Delta(\T_h)}{\dim f \ge k}} E_{f}^k(\breve\P\Lambda^k(f)),
\end{equation}
where $\E_{f}^k$ is the extension operator defined in the previous
section,
mapping into the space of Whitney forms,
while $E_{f}^k$ is an harmonic extension operator mapping
into
$\0\P\Lambda^k(\T_{f,h})$. Furthermore, the space $\breve\P\Lambda^k(f) =
\0\P\Lambda^k(f)$ if $\dim f > k$, while  
\[
\breve \P\Lambda^k(f) = \{ u \in \P\Lambda^k(f) \, |\, 
\int_f u = 0\}, \quad \text{if } \dim f = k.
\]
The decomposition \eqref{decomp} can be seen a modification 
of the more standard decompositions \eqref{decomp-basic} and
\eqref{decomp-basic-}
in the sense that we are utilizing the special extension, $\E_f^k$,
for the constant term of the traces on $f$, when $\dim f = k$.
The existence of such a decomposition of the space
$\P\Lambda^k(\T_{f})$
is an immediate consequence of
the degrees of freedom \eqref{DOF}.

As in the case $k=0$, cf. Section~\ref{ord-fcns}, the projection
$\pi_h^k$
will be constructed from a sequence of operators $\pi_{m,h}^k$, where
$\pi_h^k = \pi_{n,h}^k$. The operators $\pi_{m,h}^k$ are defined
by a recursion of the form 
\begin{equation}
\label{pimdef}
\pi_{m,h}^k = \pi_{m-1,h}^k + \sum_{f \in \Delta_m(\T_h)}
E_{f}^k \circ \tr_f \circ P_{f}^k [I - \pi_{m-1,h}^k], \qquad k \le m \le n,
\end{equation}
where the operators $P_f^k$ are local projections defined with
respect to the macroelements $\Omega_f$,
generalizing the operators $P_f^0$ introduced in
Section~\ref{ord-fcns}. 
Furthermore, the operator $\pi_{k-1,h}^k$ will be taken to be the
operator $R_h^k$ defined in Section~\ref{sec:whitney} above.
Hence, to complete the definition of $\pi_h^k$, it remains to 
give precise definitions of the local operators $E_f^k$ and $P_f^k$.

\subsection{Extension operators}
As above, to simplify the notation, we suppress the dependence on $h$
throughout the discussion.  The extension operators $E_f^k$ are
generalizations of the harmonic extension operators $E_f^0$ used for
zero forms in Section~\ref{ord-fcns}.  Let us first assume that $f \in
\Delta(\T)$ such that $f$ is not a subset of the boundary of
$\Omega_f$. In this case, the harmonic extension $E_{f}^k$ maps
$\0\P\Lambda^k(f)$ to $\0\P\Lambda^k(\T_{f})$, where $0 \le k \le \dim
f$.  More specifically, we let $E_{f}^k\phi$ be characterized by
\[
\|d E_{f}^k\phi \|_{L^2(\Omega_f)} = \inf \{\|dv\|_{L^2(\Omega_f)} \,
| \, v \in \0\P\Lambda^k(\T_{f}),\,
\tr_f v = \phi \, \}.
\]
We should note that it is a consequence of the degrees of freedom of the
spaces $\P\Lambda^k(\T_{f})$ and $\0\P\Lambda^k(\T_{f})$ that there are
feasible solutions to this optimization problem.  As a consequence, an optimal
solution exists. However, the solution is in general not unique. The solution
is only determined up to adding functions $w$ in $\0\P\Lambda^k(\T_{f,h})$
satisfying $dw =0$ on $\Omega_f$ and $tr_f w = 0$. Therefore, to
obtain 
a well defined extension operator, we
need to introduce a corresponding gauge condition.  Hence,
for any $\phi \in \0\P\Lambda^k(f)$ we let $E_{f}^k\phi \in
\0\P\Lambda^k(\T_{f})$ 
be the solution of the system
\begin{equation}
\label{local-ext1}
\begin{array}{rl}
\<E_{f}^k\phi,d\tau\>_{\Omega_f} &= 0, \quad
\tau \in N(\tr_f;\0\P\Lambda^{k-1}(\T_{f})),
\\
\<d E_{f}^k\phi,dv\>_{\Omega_f} &= 0, \quad v \in
N(\tr_f;\0\P\Lambda^{k}(\T_{f})),
\end{array}
\end{equation}
and such that $\tr_f\circ E_f^k$ is the identity on
$\0\P\Lambda^k(f)$. Here $N(\tr_f;X)$ denotes the kernel of the
operator $\tr_f$ restricted to the function space $X$.
A key property of the extension operators $E_f^k$ is that they commute
with the exterior derivative.

\begin{lem}\label{unique-ext1}
Let $f \in \Delta(\T)$.
The extension operators $E_f^k : \0\P\Lambda^k(f) \to \0\P\Lambda^k(\T_{f})$
are well defined by the system \eqref{local-ext1} for $k=0,1, \ldots
,\dim f$, and  
for $k \ge 1$ we have the identity
\begin{equation}\label{d-E-commute}
E_{f}^k d \phi= d E_{f}^{k-1}\phi, \quad \phi \in \0\P\Lambda^{k-1}(f).
\end{equation}
Moreover, the kernel of $d$ restricted to 
$N(\tr_f;\0\P\Lambda^{k}(\T_{f}))$
is $dN(\tr_f;\0\P\Lambda^{k-1}(\T_{f}))$.
\end{lem}

\begin{proof}
For $k= 0$  the first equation in the system \eqref{local-ext1}
should be omitted. The kernel of $d$ restricted to 
$N(\tr_f;\0\P\Lambda^{0}(\T_{f}))$ is just the zero function, and $E_f^0\phi$ is clearly uniquely determined by
the second equation and the extension property.
We proceed by induction on $k$. 

Assume that the statement of the
lemma holds for all levels less than $k$.
We first establish the characterization of  
the kernel of $d$, restricted to $N(\tr_f;\0\P\Lambda^{k}(\T_{f}))$.
Assume that  
$u \in N(\tr_f;\0\P\Lambda^{k}(\T_{f}))$ satisfies $du = 0$. 
Then, by the exactness of the complex $(\0\P\Lambda(\T_{f}),d)$, $u =
d\tau$ for some $\tau \in \0\P\Lambda^{k-1}(\T_{f})$. Furthermore, 
$d\tr_f \tau = \tr_f d \tau = \tr_f u = 0$. If $k=1$ this implies that
$\tau \in N(\tr_f;\0\P\Lambda^{0}(\T_{f}))$. For $k>1$
it follows from the exactness of $(\0\P\Lambda(f),d)$ that there is a
$\phi \in \0\P\Lambda^{k-2}(f)$ such that $d\phi = \tr_f\tau$.
However, the function 
\[
\sigma = \tau - dE^{k-2}\phi = \tau - E^{k-1}d\phi
\]
$\in N(\tr_f;\0\P\Lambda^{k-1}(\T_{f}))$ and satisfies $d\sigma = u$.
Hence the complex $(N(\tr_f;\0\P\Lambda(\T_{f})),d)$ is exact at
level $k$ in the sense that $dN(\tr_f;\0\P\Lambda^{k-1}(\T_{f}))$
is the kernel of $d$ restricted to $N(\tr_f;\0\P\Lambda^{k}(\T_{f}))$.

Consider a local Hodge Laplace problem  of the form:
\begin{equation}
\label{local-ext2}
\begin{array}{rl}
\<\sigma, \tau\>_{\Omega_f} - \<u,d\tau\>_{\Omega_f} &= 0, \quad
\tau \in N(\tr_f;\0\P\Lambda^{k-1}(\T_{f})),
\\
\<d\sigma, v\>_{\Omega_f} + \<d u,dv\>_{\Omega_f} &= 0, \quad v \in
N(\tr_f;\0\P\Lambda^{k}(\T_{f})),
\end{array}
\end{equation}
where the unknown $(\sigma,u) \in
N(\tr_f;\0\P\Lambda^{k-1}(\T_{f}))\times \0\P\Lambda^{k}(\T_{f})$,
and with $\tr_f u = \phi \in \0\P\Lambda^k(f)$. Since the complex
$(N(\tr_f;\0\P\Lambda(\T_{f})),d)$ is exact at
level $k$, it follows from the abstract theory of Hodge Laplace
problems, cf. for example \cite[Section 3]{bulletin}, that the system 
\eqref{local-ext2} has a unique solution. Furthermore, 
by the exactness of the same complex 
at level $k-1$, $\sigma = 0$.
Hence, $u$ and $E_f^k\phi$ satisfy the same conditions, and the
uniqueness of $E_f^k\phi$ follows by the uniqueness of $u$. 

Finally, to show the identity \eqref{d-E-commute}
we just observe that for any $\phi \in \0\P\Lambda^{k-1}(f)$,
the pair $(\sigma,u)$, with $\sigma = 0$ and 
$u= dE_f^{k-1}\phi \in \0\P\Lambda(\T_{f})$, 
satisfies the system \eqref{local-ext2} with 
$\tr_fdE_f^{k-1}\phi  = d\tr_fE_f^{k-1}\phi = d \phi$. By uniqueness
of such solutions we conclude that $dE_f^{k-1}\phi = E_f^kd\phi$.
This completes the induction argument, and the proof of the lemma.
\end{proof}


If $g \in
\Delta(\T_{f})$, with $k \le \dim g \le \dim f$ and $g \neq f$,
then $\tr_g E_{f}^k\phi  = 0$.
In the case that $f \subset \partial \Omega$, we will also have
that $f \subset \partial \Omega_f$. In this case, the definition
of the operator $E_{f}^k$ should be properly modified, such that
$E_{f}^k\phi$ is not required to be in $\0\P\Lambda^k(\T_{f})$,
but only required to be zero on the interior part of
$\partial \Omega_f$.  The key desired property is that
the extension of $E_{f}^k\phi$ from $\Omega_f$  to $\Omega$, by zero outside
$\Omega_f$, is in the global space $\P\Lambda^k(\T_h)$.

It is a consequence of the decomposition \eqref{decomp}
that any element $u$ of $\P\Lambda^k(\T)$ is uniquely determined by its
trace on $f$, $\tr_fu$, for all $f \in \Delta(\T)$ with   $\dim f \ge k$.
Furthermore, if $u$ is an element of the subspace given by
\begin{equation}\label{decomp-m}
\bigoplus_{f \in \Delta_k(\T)} \E_{f}^k(\P_0\Lambda^k(f))
+ \bigoplus_{\stackrel{f \in
  \Delta(\T)}{k \le \dim f \le m}} E_{f}^k(\breve\P\Lambda^k(f)),
\end{equation}
then $u$ is determined by $\tr_fu$ for all $f \in \Delta(\T)$ with
$k \le \dim f \le m$.
A key observation is the following.

\begin{lem}\label{invariant}
Assume that $u \in \P\Lambda^k(\T)$ belongs to the subspace given by 
\eqref{decomp-m}, where $k < m \le n$. Then its exterior derivative,
$du$, belongs to the
corresponding space 
\[
\bigoplus_{f \in \Delta_{k+1}(\T)} \E_{f}^{k+1}(\P_0\Lambda^{k+1}(f))
+ \bigoplus_{\stackrel{f \in
  \Delta(\T)}{k+1 \le \dim f \le m}}
E_{f}^{k+1}(\breve\P\Lambda^{k+1}(f)).
\]
\end{lem}

\begin{proof}
It follows from the fact that $(\P_1^-\Lambda(\T),d)$ is a complex
that
$d\E_g^k\vol_g \in \bigoplus_{f \in \Delta_{k+1}(\T)}
\E_{f}^{k+1}(\P_0\Lambda^{k+1}(f))$
for any $g \in \Delta_k(\T)$. Furthermore, if $g \in \Delta(\T)$ and
$\dim g > k$ then 
\eqref{d-E-commute} implies that 
$dE_g^k\phi  = E_g^{k+1}d\phi$ for any $\phi \in
\breve\P\Lambda^k(g)$.
As a consequence, it only remains to check terms of the form $dE_g^k\phi$,
where $\phi \in \breve\P\Lambda^k(g)$ and $\dim g = k$.

Note that $dE_g^k\phi$ is identically zero outside $\Omega_g$. 
Furthermore, consider any $f \in \Delta_{k+1}(\T)$, 
with $g \in \Delta_k(f)$. Then $\Omega_f \subset \Omega_g$ 
and the space $N(\tr_f;
\0\P\Lambda^{k}(\T_f))$ can be identified with a subspace of 
$N(\tr_g;\0\P\Lambda^{k}(\T_g))$. 
Therefore, it follows from the definition of $E_g^k\phi$ that
\[
\< dE_g^k \phi, dv \>_{\Omega_f} = 0, \quad v \in N(\tr_f;
\0\P\Lambda^{k}(\T_f)), \, f \in \Delta_{k+1}(\T), \,  g \in \Delta_k(f).
\]
However, this implies that 
\begin{align*}
dE_g^k\phi &\in \bigoplus_{\stackrel{f \in
  \Delta_{k+1}(\T)}{g \in \Delta_k(f)}}
E_{f}^{k+1}(\P\Lambda^{k+1}(f))\\
&= \bigoplus_{\stackrel{f \in
  \Delta_{k+1}(\T)}{g \in \Delta_k(f)}} \E_{f}^{k+1}(\P_0\Lambda^{k+1}(f))
+ \bigoplus_{\stackrel{f \in
  \Delta_{k+1}(\T)}{g \in \Delta_k(f)}}
E_{f}^{k+1}(\breve\P\Lambda^{k+1}(f)).
\end{align*}
This completes the proof.
\end{proof}

The harmonic extension operator just discussed is the one we will use
in the construction of the local cochain projection $\pi^k$,
cf. \eqref{pimdef}.
However, in the theory below we will also utilize an alternative local
extension, defined with respect to spaces $\P\Lambda^k(\T_f)$ instead
of $\0\P\Lambda^k(\T_f)$. For $0 \le k \le n$ the operator 
$\tilde E_f^k : \P\Lambda^k(f)
\to \P\Lambda^k(\T)$ is defined by the conditions
\begin{equation}
\label{local-ext3}
\begin{array}{rl}
\<\tilde E_{f}^k\phi,d\tau\>_{\Omega_f} &= 0, \quad
\tau \in N(\tr_f;\P\Lambda^{k-1}(\T_{f})),
\\
\<d \tilde E_{f}^k\phi,dv\>_{\Omega_f} &= 0, \quad v \in
N(\tr_f;\P\Lambda^{k}(\T_{f})),
\end{array}
\end{equation}
in addition to the extension property $\tr_f \circ \tilde E_f^k \phi =
\phi$
for all $\phi \in \P\Lambda^k(f)$.  
In complete analogy with the discussion for the operators $E_f^k$
above, by utilizing the exactness of the complex
$(\P\Lambda(\T_f),d)$
instead of the exactness of $(\0\P\Lambda(\T_f),d)$, we can conclude
with the following analog of Lemma~\ref{unique-ext1}.

\begin{lem}\label{unique-ext2}
Let $f \in \Delta(\T)$.
The extension operators $\tilde E_f^k : \P\Lambda^k(f) \to \P\Lambda^k(\T_{f})$
are well defined by the system \eqref{local-ext3} for $k=0,1, \ldots
,\dim f$, and  
for $k \ge 1$ we have the identity
\[
\tilde E_{f}^k d \phi= d \tilde E_{f}^{k-1}\phi, 
\quad \phi \in \P\Lambda^{k-1}(f).
\]
Moreover, the kernel of $d$ restricted to 
$N(\tr_f;\P\Lambda^{k}(\T_{f}))$
is $dN(\tr_f;\P\Lambda^{k-1}(\T_{f}))$.
\end{lem}

\subsection{Local projections}
Let $f \in \Delta(\T)$ and recall the definition of the spaces 
$\breve\P\Lambda^k(f)$ given above, as $\0\P\Lambda^k(f)$ if $k$ is less
than dimension of $f$, and as the subspace of $\P\Lambda^k(f)$
consisting of functions with zero mean value if $k = \dim f$.
Hence, as an alternative to \eqref{pol-complex-0}, we can state that the 
complex
\[ 
\begin{CD}
0 \to \breve\P\Lambda^0(f) @>d>> \breve\P\Lambda^{1}(f)
 @>d>> \cdots @>d>> \breve\P\Lambda^{\dim f}(f)\to 0
\end{CD}
\]
is exact. In particular, this means that the first operator, $d = d^0$,
is one--one and the last  operator, $d=d^{\dim f -1}$, is onto.
In order to define the local projections $P_f^k$, appearing in
\eqref{pimdef}, we 
will use the spaces $\breve\P(f)$ to introduce proper local spaces, 
$\breve\P\Lambda^k(\T_{f})$.   For $0 \le k < \dim f$ these
spaces lie between $\P\Lambda^k(\T_{f})$ and
$\0\P\Lambda^k(\T_{f})$, i.e.,
\[
\0\P\Lambda^k(\T_{f}) \subset \breve\P\Lambda^k(\T_{f})
\subset
\P\Lambda^k(\T_{f}).
\]
More precisely, for $0 \le k \le \dim f$ the space $\breve\P\Lambda^k(\T_{f})$
is defined by
\[
\breve\P\Lambda^k(\T_f) = \{ u \in \P\Lambda^k(\T_f) \, |\, \tr_f \in
\breve\P\Lambda^k(f) \, \},
\]
while we let $\breve\P\Lambda^k(\T_f) = \P\Lambda^k(\T_f)$ for $\dim f <
k \le n$.
We note that for $k=0$ this definition is consistent with the 
definition of the space 
$\breve \P_r\Lambda^0(\T_{f})$ used in Section~\ref{ord-fcns}.

We observe that $d\breve\P\Lambda^k(\T_{f,h}) \subset
\breve\P\Lambda^{k+1}(\T_{f,h})$. In other words,
$(\breve\P\Lambda^k(\T_f), d)$, given by 
\[
\begin{CD}
0 \to  \breve\P\Lambda^0(\T_{f}) @>d>> \breve\P\Lambda^{1}(\T_{f})
 @>d>> \cdots @>d>> \breve\P\Lambda^{n}(\T_{f}) \to 0,
\end{CD}
\]
is a complex. We also have the following:

\begin{lem}\label{hat-complex-exact}
The  complex
$(\breve\P\Lambda^k(\T_f), d)$
is exact. 
 \end{lem}

\begin{proof}
Let $m = \dim f$, and 
assume that $u \in \breve\P\Lambda^k(\T_f)$ satisfies $du = 0$.
We need to show that there is a $\sigma \in \breve\P\Lambda^{k-1}(\T_f)$ such
that $d \sigma= u$. For $k> m+1$ this follows from the exactness of
the complex $(\P\Lambda(\T_f),d)$. Assume next that  
$k \le m$. Since $d\tr_f u = \tr_f du= 0$, 
it follows from the exactness of the complex $(\breve\P\Lambda(f),d)$
that
there is 
$\phi \in \breve\P\Lambda^{k-1}(f)$ such that $d\phi = \tr_f u$.
Therefore $u - d\tilde E_f^{k-1}\phi$ is in $N(\tr_f,
\P\Lambda^k(\T_f))$ and $d(u - d\tilde E_f^{k-1}\phi) = 0$.
By Lemma~\ref{unique-ext2}, there is a 
$\tau \in N(\tr_f, \P\Lambda^{k-1}(\T_f))$ such that
$d\tau = u - d\tilde E_f^{k-1}\phi$.
Hence, the function $\sigma = \tau + \tilde E_f^{k-1}\phi$
satisfies $\tr_f \sigma = \phi \in \breve\P\Lambda^{k-1}(f)$. So
$\sigma \in \breve \P\Lambda^{k-1}(\T_f)$ and $d \sigma = u$.

Finally, we have to consider the case when 
$k= m+1$. The exactness of the complex $(\P\Lambda(\T_f),d)$
and the assumption $du=0$ implies that there is $\tau \in
\P\Lambda^{m}(\T_f)$
such that $d\tau = u$. Furthermore, the exactness of
$(\P\Lambda(f),d)$
implies that there is a $\phi \in \P\Lambda^{m-1}(f)$
such that $d\phi = \tr_f \tau$. The function 
$\sigma = \tau - d\tilde E_f^{m-1}\phi$ has vanishing trace on $f$.
Therefore, it is in $\breve\P\Lambda^m(\T_f)$, and $d\sigma = u$.
\end{proof}

We are now ready to define a local projections
$P_{f}^k : H\Lambda^k(\Omega_f) \to
\breve\P\Lambda^k(\T_{f})$ satisfying
\begin{align*}
\<P_{f}^k u,d\tau\>_{\Omega_f} &= \<u,d\tau\>_{\Omega_f}, \quad \tau \in
\breve\P\Lambda^{k-1}(\T_{f}),\\
\<d P_{f}^k u,dv\>_{\Omega_f} &= \<du,dv \>_{\Omega_f},
\quad v \in  \breve\P\Lambda^{k}(\T_{f}).
\end{align*}
The operator $P_{f}^k$ is a well defined projection onto 
$\breve\P\Lambda^k(\T_{f})$
as a consequence of Lemma~\ref{hat-complex-exact}. When $k=0$,
the space $d\P\Lambda^{-1}(\T_{f})$ should be interpreted
as the space of constants on $\Omega_f$, such that 
$P_{f}^0$ is exactly the projection defined in Section~\ref{ord-fcns}.
With this definition it is straightforward to check that
the projections $\P_f^k$ commute with the exterior derivative, i.e., 
\begin{equation}\label{d-P-commute}
P_{f}^kdu = dP_{f}^{k-1}u, \quad 0 < k \le n.
\end{equation}

\subsection{Properties of the Operators $\pi_h^k$}
The definitions of the operators $E_f^k$ and $\P_f^k$ given above
complete the construction of the operators $\pi^k = \pi_h^k$
given by the recursion \eqref{pimdef}. 
Here we shall derive two key properties of these operators,
namely that they are projections onto $\P\Lambda^k(\T)$ and that they
commute with the exterior derivative.
It is also clear from the construction that the operator 
$\pi_h^k$ is local, and, for each triangulation $\T
= \T_h$, $\pi_h^k$ is well defined as an operator on
$H\Lambda^k(\Omega)$.
However, the derivation of more precise bounds will
be delayed until the next section.

We recall that the recursion \eqref{pimdef} is initialized by choosing 
$\pi_{k-1}^k = R^k$, i.e., the special projection onto the Whitney forms
constructed in Section~\ref{sec:whitney} above.  
Therefore, we obtain from Theorem~\ref{whitney-commute} that 
\begin{equation}
\label{whitney-commute-r}
d \pi_{k-1}^k u = \pi_k^{k+1} du, \quad k=0,1, \ldots , n-1,
\end{equation}
and for $k= 0$ the two operators $\pi_{-1}^0$ and $\pi_0^0$ are
the same.
Furthermore, for functions in $\P\Lambda^k(\T)$ 
the operator $\pi_{k-1}^k$ preserves the integral of the 
trace over all subsimplexes of dimension $k$, i.e., 
\begin{equation}\label{R-key-prop-r}
\int_f \tr_f \pi_{k-1}^k u = \int_f \tr_f u, \quad f \in \Delta_k(\T),
\, u \in \P\Lambda^k(\T).
\end{equation}
In other words, if $u \in \P\Lambda^k(\T)$, then 
$(u - \pi_{k-1}^ku)|_{\Omega_f} \in \breve\P\Lambda^k(\T_f)$ for $f \in
\Delta_k(\T)$ and 
$k \ge 1$.

We observe that it follows from \eqref{pimdef} and the properties of 
the extension operators $E_f^k$, that if $f \in
\Delta_m(\T_h)$, with $m \ge k$, then
\begin{equation}\label{pi_m-recursion}
\tr_f \pi_{m}^k u = \tr_f(\pi_{m-1}^ku + P_{f}^k [u -  \pi_{m-1}^k u]).
\end{equation}
On the other hand, 
\begin{equation}\label{pi-preserve}
\tr_g \pi_m^ku = \tr_g \pi_{m-1}^ku, \quad g \in \Delta(\T), \, k \le \dim g < m.
\end{equation}
These observations are the key tools to show that $\pi^k$ is a projection.

\begin{thm}\label{projection}
The operators $\pi_h^k$ are projections onto $\P\Lambda^k(\T_h)$.
\end{thm}

\begin{proof}
Assume throughout  that $u \in \P\Lambda^k(\T)$. We have to show that
$\pi^k u = u$. We will argue that
\begin{equation}\label{pi_m-property}
\tr_f \pi_{m}^k u = \tr_f u, \quad \text{if } f \in
\Delta(\T),\quad k \le \dim f \le m,
\end{equation}
for $m = k, k+1, \ldots ,n$.
This will imply the desired result,
since functions in $\P\Lambda^k(\T)$ are uniquely determined 
by their traces on $f \in \Delta(\T)$.
We will prove \eqref{pi_m-property} by induction on $m$.
Recall that
$u - \pi_{k-1}^k u \in \breve\P\Lambda^k(\T_{f})$ for any $f \in
\Delta_k(\T)$.
As a consequence, $P_{f}^k(u - \pi_{k-1}^k u) = u - \pi_{k-1}^ku$,
and therefore \eqref{pi_m-property}, with $m= k$,  follows 
from \eqref{pi_m-recursion}.

Next, if \eqref{pi_m-property} holds for $m$ replaced by $m-1$, then
\eqref{pi-preserve} implies that
$\tr_g \pi_{m}^k u = \tr_g \pi_{m-1}^k u = \tr_g u$ for all $g \in
\Delta(\T)$, with $k \le \dim f < m$.
So it only remains to show the identity \eqref{pi_m-property}
for $f \in \Delta_m(\T)$.
However, for each $f \in \Delta_m(\T)$, we have
$(u - \pi_{m-1}^k u)|_{\Omega_f} \in \breve \P\Lambda^{k}(\T_{f})$.
Hence $P_{f}^k( u- \pi_{m-1}^k u ) =( u- \pi_{m-1}^k
u)|_{\Omega_f}$, and then \eqref{pi_m-recursion} implies that 
$\tr_f \pi_m^ku = \tr_f u$.
We have therefore verified that the operator $\pi_m^k$ satisfies
property \eqref{pi_m-property}.
This completes the proof.
\end{proof}

To show that the projections $\pi_h^k$ are cochain projections,
the following observation is useful.

\begin{lem}\label{pi-obsevation}
Assume that $0 < k \le n$ and that $u \in H\Lambda^{k-1}(\Omega)$.
For any $f \in \Delta_k(\T)$ the function 
$
d(\pi_{k-1}^{k-1}u - \pi_{k-2}^{k-1}u)|_{\Omega_f} \in
\breve\P\Lambda^k(\T_f).
$
\end{lem}

\begin{proof}
The function $e \equiv d(\pi_{k-1}^{k-1}u - \pi_{k-2}^{k-1}u)$ is
obviously in $\P\Lambda^k(\T_f)$. Therefore, it only remains to show
that
$\int_f \tr_f e = 0$. If $f = [x_0,x_1, \ldots ,x_k]$, then it follows
from the definition of $\pi_{k-1}^{k-1}$ and \eqref{stokes} that
\[
\int_f \tr_f e = \sum_{j= 0}^k(-1)^j
\int_{f_j}\tr_{f_j}P_{f_j}^{k-1}(u - \pi_{k-2}^{k-1}u) = 0.
\]
Here the last identity follows since for $\dim f_j = k-1$,
the projection $P_{f_j}^{k-1}$
projects into a space of functions of mean value zero on $f_j$.
\end{proof}

We conclude with the final result of this section.

\begin{thm}\label{commutes}
The operators $\pi_h^k$ are cochain projections, i.e., $d\pi_h^{k-1} =
\pi_h^kd$ for $k= 1,2, \ldots, n$.
\end{thm}

\begin{proof}
We will prove that for $u \in H\Lambda^{k-1}(\Omega)$, and $1 \le k \le n$,
\begin{equation}\label{d-tr_pi-commute}
\tr_f \pi_{m}^k du =  \tr_f
d\pi_{m}^{k-1} u,  \quad \text{if } f \in
\Delta(\T_h),\, k \le \dim f \le m,
\end{equation}
for $m = k,k+1, \ldots ,n$. As above, the case $m= n$ implies the desired
result.
We note that it follows from Lemma~\ref{invariant}
that if \eqref{d-tr_pi-commute} holds for any $k \le m \le n$, then
$\pi_m^kdu = d\pi_m^{k-1}u$.
 
The identity \eqref{d-tr_pi-commute} will
be established by induction on $m$, starting from $m=k$.
By \eqref{d-P-commute} and \eqref{pi_m-recursion} we have, 
for any $f \in \Delta_{k}(\T_h)$,
\[
\tr_f \pi_{k}^{k} du = \tr_f[\pi_{k-1}^{k}du
+ P_{f}^{k}(du - \pi_{k-1}^{k}du)] = d \tr_f P_{f}^{k-1}u + \tr_f(I -
P_{f}^{k})\pi_{k-1}^{k}du.
\]
On the other hand,
\begin{align*}
d \tr_f \pi_{k}^{k-1}u &=  d \tr_f P_{f}^{k-1}u + d \tr_f(I -
P_{f}^{k-1})\pi_{k-1}^{k-1} u\\
&= d \tr_f P_{f}^{k-1}u + \tr_f(I -
P_{f}^{k})d\pi_{k-1}^{k-1}u.
\end{align*}
By comparing the two expressions, and utilizing
\eqref{whitney-commute-r}, we obtain
\begin{align*}
\tr_f (\pi_{k}^{k} du -d\pi_{k}^{k-1}u) &= \tr_f (I -
P_{f}^{k})(\pi_{k-1}^kdu - d\pi_{k-1}^{k-1}u)\\
&= \tr_f (I -
P_{f}^{k})(d\pi_{k-2}^{k-1} u - d\pi_{k-1}^{k-1}u) = 0,
\end{align*}
where the last identity is a consequence of Lemma~\ref{pi-obsevation}.
So \eqref{d-tr_pi-commute} holds for $m= k$.

Assume next that \eqref{d-tr_pi-commute} holds for $m$ replaced by
$m-1$. As we observed above, this implies that 
$\pi_{m-1}^k du = d \pi_{m-1}^{k-1}u$.
Furthermore, by \eqref{pi-preserve} it follows that the operators 
$\pi_m^{k-1}$ and
$\pi_m^{k}$ satisfy \eqref{d-tr_pi-commute} for all $f \in \Delta(\T)$
with $k \le \dim f \le m-1$. Finally, 
for $f \in \Delta_m(\T)$ we have by \eqref{d-P-commute}
and \eqref{pi_m-recursion} that
\begin{align*}
\tr_f \pi_{m}^k du &= \tr_f [P_{f}^k(du - \pi_{m-1}^kdu)
+ \pi_{m-1}^k du] \\
&=  \tr_fd [P_{f}^{k-1}(u - \pi_{m-1}^{k-1}u)  +
\pi_{m-1}^{k-1}u] =  \tr_f d \pi_{m}^{k-1} u.
\end{align*}
This completes the proof.
\end{proof}

\section{Local bounds}
\label{bounds}
The purpose of this section is to derive local bounds for 
the projections $\pi_h^k$ constructed above. The main technique we
will use is scaling, a standard technique in the analysis of 
finite element methods.
The arguments below resemble parts of the discussion given
in \cite[Section 5.4]{acta}, where scaling is used in a slightly 
different setting.

>From the construction above, it follows that the operators $\pi_h^k$ are local
operators.
In fact, we observed in Section~\ref{sec:whitney} that
the operator $\pi_{k-1,h}^k= R_h^k$ has the property that 
$\tr_f \circ \pi_{k-1,h}^ku$ 
only depends on $u|_{\Omega_f^e}$.
As a consequence, $(\pi_{k-1,h}^ku)|_T$ only depends on $u$ restricted
to 
\[
\cup_{f \in \Delta_k(T)}\Omega_f^e \subset \Omega_T^e = D_{0,T}
\subset D_{k-1,T}
\]
for $T \in \T_h$ and $0 < k \le n$. Here we recall that the local 
domains  $D_{m,T}$ and $D_T = D_{n,T}$ are defined by \eqref{D-recursion}.
Therefore it follows by \eqref{D-recursion}, \eqref{pi_m-recursion}, 
and the local properties of the operators $P_f^k$ and $E_f^k$, that the
operator $\pi_h^k$ has the property that $(\pi_h^ku)|_T$ only depends
on $u|_{D_T}$ for
any $T \in \T_h$, $0 \le k \le n$. 
Furthermore, for each $h$ the operator $\pi_h^k$
is a bounded operator in $H\Lambda^k(\Omega)$. Hence, for each $h$ and 
each $T \in \T_h$ there is a constant $c=c(h,T)$ such that
\begin{equation}\label{local-bound-h}
\|\pi_h^k u \|_{L^2\Lambda^k(T)} \le c(h,T)\, (\| u \|_{L^2\Lambda^k(D_T)} + \|du
\|_{L^2\Lambda^{k+1}(D_T)}), \quad u \in H\Lambda^k(D_T).
\end{equation}
Our goal in this section is to improve this result by establishing
the uniform bound
\begin{equation}\label{local-bound-u}
\|\pi_h^k u \|_{L^2\Lambda^k(T)} \le C\, (\| u \|_{L^2\Lambda^k(D_t)} + h_T\|du
\|_{L^2\Lambda^{k+1}(D_T)}), \quad u \in H\Lambda^k(D_T),
\end{equation}
for $0 \le k \le n$, where the constant $C$ is independent of $h$ and $T$.
Since the operators $\pi_h^k$ commute with the exterior derivative,
the estimate \eqref{local-bound-u} will also imply that
\begin{equation}\label{local-bound-u-d}
\|d\pi_h^k u \|_{L^2\Lambda^k(T)} \le C \, \| du \|_{L^2\Lambda^k(D_T)},
\quad u \in H\Lambda^k(D_T),
\end{equation}
for $0 \le k <n$,
with the same constant $C$ as in \eqref{local-bound-u}. 
Therefore, the estimate \eqref{local-bound-u} will, in particular, imply
the bounds given in Theorem~\ref{bound0}.

The rest of this section will be used to prove the estimate 
\eqref{local-bound-u}.
For any fixed $T \in \T_h$, we introduce the scaling $\Phi_T(x) =
(x-x_0)/h_T$,
where $x_0$ is a vertex of $T$. We let $\hat T = \Phi_T(T)$ and $\hat
D_T = \Phi_T(D_T)$ be the corresponding reference domains
with size of order one. The restriction of the triangulation $\T_h$ 
to $D_T$ will be denoted $\T_h(D_T)$,
and $\hat \T_h(D_T)$ the induced triangulation  on $\hat D_T$.
In general we will use the hat notation to denote scaled versions of 
domains and local triangulations, for example $\hat f = \Phi_T(f)$, 
$f \in \Delta(\T_h)$.
We note that the pullback, $\Phi_T^*$ maps $H\Lambda^k(\hat
D_T)$
to $H\Lambda^k(D_T)$. Furthermore, it follows from the definition
of pullbacks that
\begin{equation}\label{pullback}
\| \Phi_T^*u \|_{L^2\Lambda^k(D)} = h_T^{-k + n/2}\| u
  \|_{L^2\Lambda^k(\hat D)}, \quad u \in L^2\Lambda^k(\hat D),
\end{equation}
where $D \subset D_T$ and $\hat D = \Phi_T(D)$.
We will obtain bounds for the
operator $\pi_h^k$, considered as a local operator mapping
$H\Lambda^k(D_T)$
to $H\Lambda^k(T)$, by studying the operator
$\Phi_T^{*-1}\pi_h^k\Phi_T^*$
as an operator mapping $H\Lambda^k(\hat
D_T)$ to $H\Lambda^k(\hat T)$.
In fact, since the since the pullbacks commute with the exterior
derivative,
it follows from \eqref{pullback}
that 
\begin{align}\label{scaling}
\|\pi_h^k u \|_{L^2\Lambda^k(T)}&= \|\Phi_T^{*-1}\pi_h^ku
\|_{L^2\Lambda^k(\hat T)}h_T^{-k+n/2}\nonumber\\
&\le \|\Phi_T^{*-1}\pi_h^k\Phi_T^{*}\|\,  
h_T^{-k+n/2}(\|\Phi_T^{*-1}u\|_{L^2\Lambda^k(\hat T)}
+ \|\Phi_T^{*-1}du\|_{L^2\Lambda^{k+1}(\hat T)})\\
&\le \|\Phi_T^{*-1}\pi_h^k\Phi_T^{*}\|\, (\|u \|_{L^2\Lambda^k(D_T)} + h_T \|du
\|_{L^2\Lambda^{k+1}(D_T)}),\nonumber
\end{align}
where $\|\Phi_T^{*-1}\pi_h^k\Phi_T^{*}\|$ denotes the operator norm  
in $\Lin(H\Lambda^k(\hat
  D_T), L^2\Lambda^k(\hat
  T))$.
Note that if we can show that this operator norm is uniformly bounded
with respect to $h$ and $T \in \T_h$, then \eqref{scaling}
will imply the desired bound
\eqref{local-bound-u}.
The following result is the key tool for this verification.

\begin{lem}\label{reference}
The operator $\Phi_T^{*-1}\pi_h^k\Phi_T^{*}$ can be identified with  
the operator $\hat \pi^k \in \Lin(H\Lambda^k(\hat
  D_T), H\Lambda^k(\hat
  T))$ obtained by constructing the operator $\pi^k$ with respect to 
the triangulation $\hat \T_h(D_T)$ of $\hat D_T$.
\end{lem}

\begin{proof}
We have to show that the operators $\pi_h^k$ and $\hat \pi^k$
satisfies $\pi_h^k \Phi_T^* = \Phi_T^* \hat \pi^k$.
In fact, the proof just consist of checking that the pullback
$\Phi_T^*$
commutes properly with the operators used to construct $\pi^k$.
A key property of the polynomial spaces
$\P\Lambda^k$ is that they are affine invariant. Therefore, in
particular, we will
have that the spaces $\P\Lambda^k(\T_h(D_T)) =
\Phi_T^*\P\Lambda^k(\hat \T_h(D_T))$.
As a consequence of this we also obtain that the local projections
$Q_{f}^k$, defined with respect to the extended macroelements
$\Omega_f^e$, satisfies
\begin{equation}\label{pullback-Q}
\Phi_T^*Q_f^k = \hat Q_f^k \Phi_T^*, \quad f \in \Delta_k(\T_h(D_T)),
\end{equation}
with the obvious interpretation of $\hat Q_f^k$ as the corresponding
projections defined with respect to the domain $\hat \Omega_f^e =
\Phi_T^*(\Omega_f^e)$. A corresponding property holds for for the
extension operators $\E_f^k$, i.e.,
$\E_f^k \phi_T^* = \Phi_T^* \hat \E_f^k$,
where $\hat \E_f^k$ maps $\P_0\Lambda^k(\hat f)$ to
$\0\P_1^-\Lambda^k(\hat \T_{f,h})$.
In particular,
\begin{equation}\label{pullback-vol}
\E_f^k \vol_{f} = \Phi_T^* \hat \E_f^k \Phi_T^{*-1}\vol_f = \Phi_T^*
\hat \E_f^k \vol_{\hat f}.
\end{equation}
 
Consider the operator $S_h^0\Phi_T^*$, where $S_h^k$ are the
operators introduced in Section~\ref{sec:whitney} above. By
\eqref{pullback-vol} we have, for any $u \in
H\Lambda^k(\hat D_T)$, 
\begin{align}\label{pullback-S0}
S_h^0\Phi_T^*u &= \sum_{f \in \Delta_{0}(\T_h(D_T))} 
\left(\int_{\Omega_f} \Phi_T^*u \wedge \vol_{\Omega_f}
\right) \E_f^{0}\vol_f\nonumber\\
&= \sum_{f \in \Delta_{0}(\T_h(D_T))} 
\left(\int_{\Omega_f} \Phi_T^*(u \wedge \Phi_T^{*-1}\vol_{\Omega_f}
\right) \E_f^{0}\vol_f \\
&= \sum_{f \in \Delta_{0}(\T_h(D_T))} \left(\int_{\hat \Omega_f} (u
  \wedge \vol_{\hat \Omega_f}
\right) \Phi_T^* \hat \E_f^{0}\vol_{\hat f} = \Phi_T^* \hat S^0u.\nonumber
\end{align}
In general, we define the operators $\hat S^k$ with respect to 
the reference domain $\hat D_T$ as outlined in
Section~\ref{sec:whitney}.
In particular, the weight
functions $\hat z_f^k$ are taken to be $\Phi_T^{*-1}z_f^k$.
It follows essentially from \eqref{pullback-Q}, and an argument 
similar to one leading to \eqref{pullback-S0},
that $S_h^k \Phi_T^* = \Phi_T^* \hat S^k$, and this further leads to 
\begin{equation}\label{pullback-R}
\pi_{k-1,h}^k \Phi_T^* = R_h^k \Phi_T^* = \Phi_T^* \hat R_h^k = \Phi_T^*
\hat \pi_{k-1,h}^k.
\end{equation}

It is also straightforward to check that the local projections 
$P_f^k$ and the extension operators $E_f^k$
satisfy the corresponding properties $P_f^k \Phi_T^* = \Phi_T^* \hat P_f^k$
and $E_f^k \Phi_T^* = \Phi_T^* \hat E_f^k$, which implies that
\[
E_f^k \tr_f P_f^k \Phi_T^* = \Phi_T^* \hat E_f^k \tr_{\hat f}\hat
P_f^k.
\]
By combining this with the recursion \eqref{pi_m-recursion} and 
\eqref{pullback-R}, we obtain
the relation $\pi_{m,h}^k \Phi_T^* = \Phi_T^* \hat \pi_m^k$ for $k
\le m \le n$. In particular, the desired relation $\pi_h^k \Phi_T^* =
\Phi_T^* \hat \pi^k$ is obtained for $m=n$.
\end{proof}
 
We now have the following main result of this section.

\begin{thm}\label{bound}
The operators $\pi_h^k$ satisfy the bounds \eqref{local-bound-u}
and \eqref{local-bound-u-d}, where the constant C is independent of 
$h$ and $T \in \T_h$.
\end{thm}

\begin{proof}
It follows from \eqref{local-bound-h} that for each $h$ and $T$,
there is constant $C(h,T)$ such that
\begin{equation}\label{local-bound-h2}
\|\hat \pi^k u \|_{L^2\Lambda^k(\hat T)} \le C(h,T)\| u
\|_{H\Lambda^k(\hat D_T)}, \quad u \in H\Lambda^k(\hat D_T),
\end{equation}
where, as above,  $\hat \pi^k$ is obtained by constructing the
operator $\pi^k$ with respect to 
the triangulation $\hat \T_h(D_T)$ of $\hat D_T$.
However, due to the assumption of shape regularity of the family
$\{\T_h\}$,
it follows that the induced triangulations $\hat \T_h(D_T)$ varies over a
compact set. Therefore, the constant $C(h,T)$ is uniformly
bounded with respect to $h$ and $T \in \T_h$.
The desired estimate \eqref{local-bound-u} now follows from 
Lemma~\ref{reference}, combined with \eqref{scaling}
and \eqref{local-bound-h2}. Finally, 
as we observed above, \eqref{local-bound-u-d} follows from
\eqref{local-bound-u}
and the fact that the projections $\pi_h^k$ commutes with $d$.
\end{proof} 

Finally, we observe that since the since shape regularity of the
triangulation $\{\T_h \}$ implies that the the covering $\{D_T \}$ 
of $\Omega$ has
a bounded overlap property, it follows from the bounds 
\eqref{local-bound-u} and \eqref{local-bound-u-d} that the global
estimates 
\[
\|\pi_h^k u \|_{L^2(\Omega)} \le C\, (\| u \|_{L^2(\Omega)} + h\|du
\|_{L^2(\Omega)})
\]
and 
\[
\|d\pi_h^k u \|_{L^2(\Omega)} \le C\, 
(\| du \|_{L^2(\Omega)}, \quad u \in H\Lambda^k(\Omega),
\]
where $C$ is independent of $h$, also holds.

\subsection*{Acknowledgement} The second author is grateful to Snorre
H. Christiansen for many useful discussions.

\end{document}